\documentclass[11pt,eqno]{article}
\usepackage{amssymb}
\usepackage[leqno]{amsmath}
\usepackage[pdftex]{graphicx}
\usepackage{epsfig}
\usepackage{epstopdf}
\usepackage{algorithm}
\usepackage{algorithmic}
\usepackage{array}
\usepackage{caption,subcaption}
\usepackage{color}
\usepackage{float}
\usepackage{multirow}
\usepackage{footmisc}
\usepackage{picinpar}
\usepackage{longtable,pdflscape}

\RequirePackage{xcolor}[2.11]
\colorlet{siaminlinkcolor}{green!50!black}
\colorlet{siamexlinkcolor}{blue!50!black}
\colorlet{siamreviewcolor}{black!50}

\usepackage[colorlinks,
            linkcolor=siaminlinkcolor,
            anchorcolor=green,
            citecolor=siaminlinkcolor
            ]{hyperref}
\usepackage{fancyhdr}
\usepackage{rotating}
\usepackage[noabbrev,capitalise,nameinlink]{cleveref}

\makeatletter
\@addtoreset{footnote}{section}
\usepackage{indentfirst}
\graphicspath{ {./Figures/} }

\begin{document}

\def\N{\mathbb{N}}
\def\R{\mathbb{R}}
\def\bc{{\rm{c}}}
\def\ba{{\rm a}}
\def\bb{{\rm b}}
\def\be{{\rm{e}}}
\def\bd{{\rm d}}
\def\bh{{\rm h}}
\def\bw{{\rm{w}}}
\def\bx{{\rm x}} 
\def\by{{\rm y}}
\def\bz{{\rm z}}
\def\bq{{\rm q}}
\def\ts{$\tau$-stationary point}
\def\P{\mathrm{Prox}_{\tau\lambda\|\cdot\|_0}}
\def\nzno{\rm{NL0R}}
\def\supp{\mathrm{supp}}

\newcommand{\qed}{\hphantom{.}\hfill $\Box$\medbreak}
\newcommand{\proof}{\noindent{\bf Proof \ }}
\newcommand{\sign}{{\mathrm{sgn}}}
\newcommand{\tzk}{$\mathbb{T}_{\tau}(\bfz^k;s)$}
\newcommand{\beq}{\begin{eqnarray}}
\newcommand{\eeq}{\end{eqnarray}}
\newcommand{\erhao}{\fontsize{8pt}{\baselineskip}\selectfont}

\numberwithin{equation}{section}
\newtheorem{theorem}{Theorem}[section]
\newtheorem{lemma}{Lemma}[section]
\newtheorem{proposition}{Proposition}[section]
\newtheorem{corollary}{Corollary}[section]
\newtheorem{definition}{Definition}[section]
\newtheorem{remark}{Remark}[section]
\newtheorem{example}{Example}[section]
\newtheorem{assumption}{Assumption}[section]
\newenvironment{cproof}
{\begin{proof}
 [Proof.]
 \vspace{-3.2\parsep}}
{\renewcommand{\qed}{\hfill $\Diamond$} \end{proof}}

\renewcommand{\theequation}{
\arabic{equation}}
\renewcommand{\thefootnote}{\fnsymbol{footnote}}


\thispagestyle{empty}

\begin{center}
\topskip10mm
\LARGE{\bf Newton Method for $\ell_0$-Regularized Optimization}
\end{center}
\begin{center}
\setcounter{footnote}{0}
Shenglong Zhou\\ shenglong.zhou@soton.ac.uk\\
School of Mathematics, University of Southampton, UK\\
Lili Pan\\ panlili1979@163.com\\
Department of Mathematics, Shandong University of Technology,  China\\
Naihua Xiu\\ nhxiu@bjtu.edu.cn\\
 Department of Applied Mathematics, Beijing Jiaotong University,  China\\
{\small
}
\end{center}

\vskip12pt

\begin{abstract}\noindent  As a tractable approach, regularization is frequently adopted in sparse optimization. This gives rise to the regularized optimization, aiming at minimizing the $\ell_0$ norm or its continuous surrogates that characterize the sparsity.  From the continuity of surrogates to the discreteness of $\ell_0$ norm, the most challenging model is the  $\ell_0$-regularized optimization. To conquer this hardness,  there is a vast body of work on developing numerically effective methods. However, most of them only enjoy that either the (sub)sequence converges to a stationary point from the deterministic optimization perspective or the distance between each iterate and any given sparse reference point is bounded by an error bound in the sense of probability. In this paper, we develop a Newton-type method for the $\ell_0$-regularized optimization and prove that the generated sequence converges to a stationary point globally and quadratically under the standard assumptions, theoretically explaining that our method is able to perform surprisingly well.

\vspace{3mm}
\noindent{\bf \textbf{Keywords}:}
$\ell_0$-regularized optimization,  \ts,  Newton method, Global and quadratic convergence
 
\vspace{3mm}
\noindent{\bf \textbf{Mathematical Subject Classification}:}  65K05$\cdot$ 90C46$\cdot$ 90C06$\cdot$ 90C27 
\end{abstract}
{}
\numberwithin{equation}{section}

\section{Introduction}
\noindent Over the last decade, sparsity has been thoroughly investigated due to its {extensive}
applications ranging from compressed sensing \cite{donoho2006compressed, candes2006robust, candes2005decoding}, signal and image processing \cite{elad2010role, elad2010sparse, chen2012non, bian2015linearly}, machine learning \cite{wright2010sparse, yuan2012visual}  to neural networks \cite{bian2012smoothing, lin2019toward, din2020sparsity} lately.
Sparsity is frequently characterized by $\ell_0$ norm  and its penalized problem is commonly phrased as $\ell_0$-regularized optimization, taking the form of
\begin{equation}\label{L0O}
\min_{\bx\in \mathbb{R}^n }~~ f(\bx) + \lambda\|\bx\|_0,
\end{equation}
where $f:\mathbb{R}^n\rightarrow\mathbb{R}$ is twice continuously differentiable and bounded from below, $\lambda>0$ is the penalty parameter and $\|\bx\|_0$ is $\ell_0$  norm of $\bx$, counting the number of non-zero elements of $\bx$.
Differing from the regularized optimization, another category of sparsity involved  {problems} that have been well studied is the so-called sparsity constrained optimization:
\begin{equation}\label{SCO}
\min_{\bx\in \mathbb{R}^n }~ f(\bx), ~~~{\rm s.t.}~ \|\bx\|_0\leq s,
\end{equation}
where $s\leq n$ is a given positive integer. Based on the two optimizations, large numbers of state-of-the-art methods  have been proposed  {in the last decade}. In particular, many of them are designed for a special application,  {compressed sensing} (CS), where the least squares are taken into account, namely
\begin{eqnarray}\label{cs}
 f(\bx) := f_{cs} (\bx) \equiv  \|A\bx-\by\|^2.
\end{eqnarray}
Here, $A\in\R^{m\times n}$ is the  sensing matrix and $\by\in\R^{m}$ is the measurement.
\subsection{Selective Literature Review}\label{subsec:lit-rev} 
\noindent Since there is a vast body of work developing numerical methods to solve the \eqref{SCO} or \eqref{L0O},  we present a brief overview of work that is able to clarify our motivations of this paper.

{\bf
(a) Methods for \eqref{SCO}} are known as greedy ones.  For the case of CS, one can refer to
 orthogonal matching \cite[OMP]{pati1993orthogonal,tropp2007signal}, gradient pursuit \cite[GP]{blumensath2008gradient},
compressive sample matching pursuit  \cite[CoSaMP]{needell2009cosamp}, subspace pursuit \cite[SP]{dai2009subspace}, normalized iterative hard-thresholding \cite[NIHT]{blumensath2010normalized}, hard-thresholding pursuit \cite[HTP]{foucart2011hard} and accelerated iterative hard-thresholding \cite[AIHT]{blumensath2012accelerated}.
Methods for the general model   \eqref{SCO}  include the gradient support pursuit \cite[GraSP]{bahmani2013greedy}, iterative hard-thresholding \cite[IHT]{beck2013sparsity}, Newton gradient pursuit \cite[NTGP]{yuan2014newton},  conjugate gradient iterative hard-thresholding \cite[CGIHT]{blanchard2015cgiht},
gradient hard-thresholding pursuit \cite[GraHTP]{yuan2017gradient}, improved iterate hard-thresholding  \cite[IIHT]{pan2017convergent} and  Newton hard-thresholding pursuit \cite[NHTP]{zhou2019global}.

To derive the convergence results, most methods enjoy the theory that the distance between each iterate to any given reference (sparse) point is bounded by an error through statistic analysis. By contrast, methods like IHT, IIHT and NHTP have been proved to converge to a stationary point globally in the sense of the deterministic way. Moreover, if Newton directions are interpolated into some methods, for example, CoSaMP, SP, GraSP, NTGP and GraHTP, then their demonstrated empirical  {performances are} extraordinary in terms of super-fast computational speed and high order of accuracy, but without  deterministic theoretical guarantees for a long time. Until recently,  authors in \cite{zhou2019global} first proved that their proposed NHTP has global and quadratic convergence properties, which unravel the reason  why these methods behave exceptionally well.

{\bf (b) Methods for \eqref{L0O}} aiming at addressing CS  {problem} via the model \eqref{L0O} include iterative hard-thresholding algorithm \cite[IHT]{blumensath2008iterative}, continuous exact $\ell_0$ penalty \cite[CEL0]{soubies2015continuous}, two methods: continuation single best replacement and
$\ell_0$-{regularization} path descent  in \cite[CSBR, L0BD]{soussen2015homotopy},
forward-backward splitting \cite[FBS]{attouch2013convergence},
 extrapolated proximal  {iterative hard-thresholding} algorithm \cite[EPIHT]{bao2016image} and
mixed integer optimization method \cite[MIO]{bertsimas2016best},
to name just a few.
While for the general problem   \eqref{L0O}, one can see penalty decomposition \cite[PD]{lu2013sparse} where equality and inequality constraints  {are} also considered,
iterative hard-thresholding \cite[see]{lu2014iterative} where the box and convex cone  {are} taken into account,
proximal gradient  method and coordinate-wise support optimality method \cite[PG, CowS]{beck2018proximal} where sparse solutions  {are} sought from a symmetric set, random proximal alternating minimization method \cite[RPA]{Patrascu2015A}, active set Barzilar-Borwein \cite[ABB]{Wanyou2019An} and {{a very recently smoothing proximal
gradient method  \cite[SPG]{bian2020smoothing}.}}
 Note that these methods can be regarded as the first-order methods since they only benefit from the first-order information such as  gradients or function values.  Then second-order methods have attracted much attention lately, including primal dual active set \cite[PDAS]{ito2013variational},
primal dual active set with continuation \cite[PDASC]{jiao2015primal} and
  support detection and root finding \cite[SDAR]{Huang2018Constructive}.

As for convergence results,  either error bounds are achieved for methods such as IHT, EPIHT, PDASC and SDAR, or a subsequence converges to a stationary point (which is a local convergence property) for methods like PD, PG and  ABB. It is worth mentioning that authors in \cite{attouch2013convergence}  {prove} that FBS converges to a critical point globally and {{authors \cite[SPG]{bian2020smoothing} also  {show} the global convergence to a relaxation problem of \eqref{L0O}}}. Apart from that, no better deterministic theoretical guarantees (like quadratic convergence) have been established on algorithms for solving \eqref{L0O}. Therefore, a natural question is: can we develop an algorithm based on $\ell_0$-regularized optimization that enjoys the global and quadratic convergence?
\subsection{Contributions} 
\noindent To answer the above question, we first introduce a \ts, an optimality condition of \eqref{L0O}, and then reveal its relationship with local/global minimizers by  { \cref{relation1}}.  It is known that a \ts\ is a necessary optimality condition by  \cite[Theorem 4.10]{beck2018proximal}. However, we show that it is also a sufficient condition under the assumption of strong convexity.

The \ts\ can be expressed as a stationary equation system \eqref{station}, and allows us to employ the Newton-type method dubbed as \nzno, an abbreviation for  Newton method for $\ell_0$-regularized optimization \eqref{L0O}. Differing from the classical Newton methods that are usually employed on continuous equation systems, the stationery equation system turns out to be discontinuous. Despite that, we succeed in establishing the global and quadratic convergence properties for \nzno\ under standard assumptions, see \cref{converge-rate}.    As far as we know, it is the first paper that establishes both properties for an algorithm aiming at solving the $\ell_0$-regularized optimization problem.

Finally, extensive numerical experiments are conducted in this article and demonstrate that \nzno\ is very competitive when benchmarked against a number of leading solvers for solving the compressed sensing and sparse complementarity problems. In a nutshell, it is capable of delivering relatively accurate sparse solutions with fast computational speed.

It is worth mentioning that, PDASC, SDAR and NHTP also adopt the idea of the \ts. The former two always set $\tau=1$, while similar to NHTP, \nzno\ benefits from more choices of $\tau$. In addition,  the gradient direction and Amijio-type rule of updating the step size are integrated. Those strategies are alternatives if the Newton direction does not guarantee a sufficient decline of the objective function values during the process. By contrast, PDASC and SDAR only take advantage of the Newton directions with unit step sizes. Therefore, they are hard to establish the global convergence results.  Now, for the method NHTP  aiming at tackling  \eqref{SCO},   the sparsity level $s$ is required, but is usually unknown and somehow decides the quality of the final solutions. In \eqref{L0O}, the parameter $\lambda$ also plays an important role in pursuing sparse solutions. We will show that $\lambda$ is able to be set up in a proper range and the proposed method  \nzno\ could effectively tune it adaptively in numerical experiments.

\subsection{Organization and Notation}
\noindent The rest of the paper is organized as follows. Next section establishes the optimality conditions of \eqref{L0O}  {with the help of} the \ts\ whose relationship with the local/global minimizers of \eqref{L0O} by  \Cref{relation1} is also given. In \Cref{Sect:Newton-Method-BLS}, we design the  { Newton-type method} for the $\ell_0$-regularized optimization (\nzno), followed by the main convergence results including the support set identification, global and quadratic convergence properties under some standard assumptions. Extensive numerical experiments are presented in  \Cref{Sec: numerical exp}, where the implementation of \nzno\ as well as its comparisons with some other excellent solvers for solving problems, such as compressed sensing and sparse complementarity problems, are provided. Concluding remarks are made in the last section.

We end this section with some notation to be employed throughout the paper. Let $\N_n:=\{ 1,2,\cdots,n\}$. Given a vector $\bx$, let $|\bx|:=(|x_1,|x_2|,\cdots,|x_n|)^\top,$   $\|\bx\|^2:= \sum_ix_i^2$  be its  $\ell_2$ norm. 
The support set of $\bx$ is $\supp(\bx)$ consisting of indices of its non-zero elements.  Given a set $T\subseteq  \N_n$,  $|T|$ and $\overline{T}$ are the cardinality and the complementary set. The sub-vector of $\bx$  containing elements  indexed on $T$ is denoted by $\bx_T\in\mathbb{R}^{|T|}$. Next, $\lceil a\rceil$ stands for the smallest integer that is no less than $a$. Now, for a matrix $A\in\mathbb{R}^{m\times n}$, let $\|A\|_2$ represent  its spectral norm, i.e., its maximum singular value. Write $A_{T,J}$ is the sub-matrix containing rows indexed on ${T}$ and columns indexed on $J$. In particular, denote the sub-gradient and  sub-Hessians by
\begin{eqnarray*}
\nabla_T f(\bx)&:=&(\nabla f(\bx))_{T},~~~~~~~\nabla_T^2f(\bx)~:=(\nabla^2f(\bx))_{T,T},\\
\nabla_{T,J}^2f(\bx)&:=&(\nabla^2f(\bx))_{T,J},~~~~\nabla_{T:}^2f(\bx):=(\nabla^2f(\bx))_{T,\N_n}.
\end{eqnarray*}
\section{Optimality} \label{Sect:Optimality}
\noindent Some necessary optimality conditions of  \eqref{L0O}  have been studied. These include ones in \cite[Theorem 2.1]{lu2013sparse} and \cite[Theorem 4.10]{beck2018proximal}. Here, inspired by the latter, we introduce a $\tau$-stationary point (this is the same as the $L$-stationarity in \cite{beck2018proximal}).
\subsection{\ts}
\noindent  A vector $\bx \in \R^n $ is called a \ts\ of  \eqref{L0O}
if there is a $\tau> 0$ such that
\begin{eqnarray}
\label{tau-eq}
 \bx &\in &\P\left(\bx-\tau \nabla  f (\bx)\right):=\underset{\bz\in\R^n}{\rm argmin} ~\frac{1}{2} \|\bz-(\bx-\tau \nabla  f (\bx))\|^2 + \tau\lambda\|\bz\|_0. 
\end{eqnarray}
It follows from \cite{attouch2013convergence} that the operator $\P(\bz)$ takes a closed form as
\begin{eqnarray}\label{tau-sta-exp}
\Big[\P\left(\bz\right)\Big]_i =\left\{\begin{array}{ll}
z_i,&~~ |z_i|>\sqrt{2\tau\lambda},\\
\{z_i,0\},&~~ |z_i|=\sqrt{2\tau\lambda},\\
0,&~~ |z_i|<\sqrt{2\tau\lambda}.
\end{array}\right.
\end{eqnarray}
This allows us to characterize a \ts\ by conditions below equivalently, see \cite[ Theorem 24]{tropp2006just}  and \cite[Lemma 2]{blumensath2008iterative}.
\begin{lemma}\label{sts}A point $\bx$ is a \ts\ with $ \tau >0$ of  \eqref{L0O} if and only if
\begin{eqnarray}\label{tau-sta-cond}
 \begin{cases}
\nabla_i f(\bx) = 0~{\rm and}~|x_i|\geq\sqrt{2\tau\lambda},&  i\in\supp(\bx),\\
|\nabla_i f(\bx)| \leq\sqrt{2\lambda/\tau},&  i\notin\supp(\bx).
\end{cases} 
\end{eqnarray}
\end{lemma}
From \Cref{sts}, for any $0<\tau_1\leq \tau$,  a \ts\ $\bx$   is also a $\tau_1$-stationary point due to  $ 2\tau\lambda \geq 2\tau_1\lambda $ and $ 2\lambda/\tau \leq 2\lambda/\tau_1 $. Our next major result needs the strong smoothness and convexity of $f$.
\begin{definition}A function $f$ is  strongly smooth with a constant $L>0$ if
\begin{equation}\label{strong-smooth}
f(\bz)\leq f(\bx)+\langle \nabla f(\bx), \bz-\bx\rangle+({L}/{2})\|\bz-\bx\|^2, ~\forall~ \bx,\bz\in\R^n.\end{equation}
A function $f$ is strongly convex with a constant $\ell >0$ if
\begin{equation}\label{strong-convexity}
f(\bz)\geq f(\bx)+\langle \nabla f(\bx), \bz-\bx\rangle+({\ell }/{2})\|\bz-\bx\|^2,~ \forall~ \bx,\bz\in\R^n.\end{equation}
We say  a function $f$ is locally strongly convex with a constant $\ell >0$ around $\bx$ if \eqref{strong-convexity} holds for any point  $\bz$ in the neighbourhood of $\bx$.\end{definition} Something needs emphasize here is that when the function is locally strongly convex, the constant $\ell $ depends on the point $\bx$. We drop the dependence for simplicity since it would not cause confusion in the context. The strong convexity and  smoothness respectively indicate that, for any $\bx,\bz\in\R^n$
\begin{eqnarray}
\label{strong-convexity-grad}
\ell \|\bz-\bx\|\leq\|\nabla f(\bz) - \nabla f(\bx)\|   \leq L\|\bz-\bx\|.
\end{eqnarray}
\subsection{First order optimality conditions}
\noindent Our next major result is to establish the relationships between a \ts\  and a local/global  minimizer of \eqref{L0O}.
\begin{theorem}\label{relation1}
For problem \eqref{L0O}, the following results hold.
\begin{itemize}
\item[{\rm 1)} ] {\bf (Necessity)} A global minimizer $\bx^*$ is also a \ts\ for any $0<\tau< 1/L$ if $f$ is strongly smooth with $L>0$. Moreover,
\begin{equation}\label{strong-sta-eq}
\bx^*=\P\left(\bx^*-\tau \nabla  f (\bx^*)\right).
\end{equation}
\item[{\rm 2)}] {\bf(Sufficiency)}
A \ts\ with $\tau >0 $ is a local minimizer if  $f$ is convex. Furthermore,
a \ts\ with $\tau(>)\geq1/\ell  $ is also a (unique) global minimizer if  $f$ is strongly convex  with $\ell >0$.
\end{itemize}
\end{theorem}
\begin{proof}
1) Denote $\mathbb{P}:=\P\left(\bx^*-\tau \nabla  f (\bx^*)\right)$ and $\mu:=L-1/\tau<0$ due to $0<\tau< 1/L$.  Let $\bx^*$ be a  global minimizer and consider any point $\bz\in\mathbb{P}$. Then we have   
\begin{eqnarray*}
&& 2f (\bz) +2\lambda\|\bz\|_0\\
 & {\leq} &  2f (\bx^*)+2\langle \nabla  f (\bx^*),\bz-\bx^* \rangle +  L \|\bz-\bx^*\|^2 +2\lambda\|\bz\|_0\\
&=&  2f (\bx^*)+2\langle \nabla  f (\bx^*),\bz-\bx^* \rangle +  ({1}/{\tau})  \|\bz-\bx^*\|^2 +   {\mu} \|\bz-\bx^*\|^2+ 2\lambda\|\bz\|_0\\
&=&  2f (\bx^*)+ ({1}/{ \tau})  \|\bz-(\bx^*-\tau \nabla  f (\bx^*))\|^2  -  {\tau}  \|\nabla  f (\bx^*)\|^2+ 2\lambda\|\bz\|_0+   {\mu}  \|\bz-\bx^*\|^2\\
& {\leq} &  2f (\bx^*)+  ({1}/{\tau})  \|\bx^*-(\bx^*-\tau \nabla  f (\bx^*))\|^2 + 2\lambda\|\bx^*\|_0-  {\tau}  \|\nabla  f (\bx^*)\|^2+   {\mu}  \|\bz-\bx^*\|^2\\
&=&  2f (\bx^*)+2\lambda\|\bx^*\|_0+   {\mu} \|\bz-\bx^*\|^2\\
 & {\leq}&   2f (\bz)+2\lambda\|\bz\|_0 +  {\mu}  \|\bz-\bx^*\|^2,
\end{eqnarray*}
where the first, second and third inequalities hold respectively from the facts that $f$ being strongly smooth,
$\bz\in\mathbb{P}$
and $\bx^*$ being the global minimizer of \eqref{L0O}.
This together  with $\mu<0$ leads to
$0\leq (\mu/2) \|\bz-\bx^*\|^2<0,$
which yields $\bz=\bx^*$. Therefore, $\bx^*$ is a \ts\ of \eqref{L0O}. Since $\bz$ is arbitrary in $\mathbb{P}$ and $\bz=\bx^*$,  $\mathbb{P}$ is a singleton only containing  $\bx^*$.

2) Let $\bx^*$ be a \ts\ with $\tau>0$ with $T_*:=\supp(\bx^*)$. Consider a neighbour region of $\bx^*$ as $N(\bx^*)=\left\{\bx\in\R^n: \|\bx-\bx^*\|< \epsilon_*\right\}$, where
\begin{eqnarray*}\epsilon_*:= 
\begin{cases}
\min\left\{\min_{i\in T_*} |x^*_i|,\sqrt{\tau\lambda/(2n)}\right\},& \bx^*\neq 0,\\
\sqrt{\tau\lambda/(2n)},& \bx^*=0.
\end{cases}  \end{eqnarray*}
For any point $\bx\in N(\bx^*)$, we conclude $T_* \subseteq \supp(\bx)$. In fact, this is true when $\bx^*=0$. When $\bx^*\neq0$, if there is a $j$ such that $j\in T_*$ but $j\notin \supp(\bx) $, then we derive a contradiction:
 $$\epsilon_* \leq \min_{i\in T_*} |x^*_i| \leq |x_j^*|=|x_j^*- x_j|\leq \|\bx-\bx^*\| < \epsilon_*. $$
Therefore, we have  $T_* \subseteq \supp(\bx)$. The convexity of $f$ suffices to
\begin{eqnarray}
 f (\bx)-f (\bx^*) &\geq &  \langle \nabla  f (\bx^*),\bx-\bx^* \rangle \nonumber\\
 &= &   \langle \nabla_{T_*}  f (\bx^*),(\bx-\bx^*)_{T_*}  \rangle +\langle \nabla_{\overline T_*}  f (\bx^*),(\bx-\bx^*)_{\overline T_*}  \rangle\nonumber\\
\label{ffp}&\overset{\eqref{tau-sta-cond}}{=}&
 \langle \nabla_{\overline T_*}  f (\bx^*),\bx_{\overline T_*}  \rangle =:  \phi.
\end{eqnarray}
If  $T_* =\supp(\bx)$, then $\phi=0$ due to $\bx_{\overline T_*}=0$ and $\|\bx^*\|_0=\|\bx\|_0$. These allow us to derive that
\begin{eqnarray*}
 f (\bx) +\lambda\|\bx\|_0  \overset{\eqref{ffp}}{\geq}     f (\bx^*) + \phi +\lambda\|\bx\|_0 = f (\bx^*) + \lambda\|\bx^*\|_0.
\end{eqnarray*}
If  $T_* \subseteq (\neq) \supp(\bx)$, then $\|\bx\|_0-1\geq\|\bx^*\|_0$. In addition,
\begin{eqnarray*}\phi&= &\langle \nabla_{\overline T_*}  f (\bx^*),\bx_{\overline T_*}  \rangle  \geq -\|\nabla_{\overline T_*}  f (\bx^*)\|\|\bx_{\overline T_*} \|\\
 &\overset{\eqref{tau-sta-cond}}{\geq}& - \sqrt{|\overline T_*|2\lambda/\tau} \|\bx_{\overline T_*} -\bx_{\overline T_*}^*\| \geq  -  \sqrt{n2\lambda/\tau} \epsilon_*>-\lambda.
\end{eqnarray*}
These facts enable us to derive that
\begin{eqnarray*}
 f (\bx) +\lambda\|\bx\|_0 &\overset{\eqref{ffp}}{\geq} &    f (\bx^*) + \phi +\lambda\|\bx\|_0 \\
 &>& f (\bx^*) +\lambda \|\bx\|_0-  \lambda\\
 &\geq& f (\bx^*) +\lambda\|\bx^*\|_0.
\end{eqnarray*}
Both cases show  the local optimality of $\bx^*$ in the region $N(\bx^*)$. Again, it follows from $\bx^*$ being a \ts\ with $\tau>0$ that
$$({1}/{2})\|\bx-(\bx^*-\tau \nabla  f (\bx^*))\|^2 +\tau\lambda\|\bx\|_0 \geq({1}/{2})\|\bx^*-(\bx^*-\tau \nabla  f (\bx^*))\|^2 +\tau\lambda\|\bx^*\|_0,$$
for any $\bx\in \R^n$, which suffices to
\begin{equation}\label{projected-1}
\langle \nabla  f (\bx^*), \ \bx-\bx^* \rangle +\lambda\|\bx\|_0\geq -({1}/({2\tau}))\|\bx-\bx^*\|^2+\lambda\|\bx^*\|_0 .
\end{equation}
Since $f$ is strongly convex, for any $\bx \neq\bx^*$, we have
 \begin{eqnarray*}
 f (\bx)+\lambda\|\bx\|_0&\overset{\eqref{strong-convexity}}{\geq} &  f (\bx^*)+\langle \nabla  f (\bx^*),\bx-\bx^* \rangle+({\ell }/{2})\|\bx-\bx^*\|^2+\lambda\|\bx\|_0\\
&\overset{\eqref{projected-1}}{\geq}&
f (\bx^*)+(({\ell  -1/\tau})/{2})\|\bx-\bx^*\|^2+\lambda\|\bx^*\|_0\\  
 &\geq&   f (\bx^*)+\lambda\|\bx^*\|_0,
\end{eqnarray*}
where the last inequality is from  $\tau\geq1/\ell $. Clearly, if $\tau>1/\ell $, then the last inequality holds strictly, which means $\bx^*$ is a unique global minimizer. \qed \end{proof}
\noindent Let us consider an example to illustrate the above theorem.
\begin{example}Let $\ba=(t~1~1)^\top$, $\lambda>                                                                                                                                                                                                                                                                                                                                                                                                                                                                                                                                                                                                                                                                                                                                                                                                                                                                                                                                                                                                                                                                                                                                                                                                                                                                                                                                                                                                                                                                                                                                                                                                                                                                                                                                                                                                                                                                                                                                                                                                                                                                                                                                                                                                                                                                                                                                                                                                                                                                                                                                                                                                                                                                                                                                                                                                                                                                                                                                                                                                                                                                                                                                                                                                                                                                                                                                                                                                                                                                                                                                                                                                                                                                                                                                                                                                                                                                                                                                                                                                                                                                                                                                                                                                      8$ and $f$ be given by
\begin{equation}\label{ex-2}
f(\bx):= \frac{1}{2}(\bx-\ba)^\top\left[\begin{array}{rrr}
2&0&0\\
0&3&1\\
0&1&3\\
\end{array}\right](\bx-\ba).
\end{equation}
\end{example}
It is easy to verify that $f$ is strongly smooth with $L=2$ and also strongly convex with $\ell =1$. Consider a point $\bx^*=(t~0~0)^\top$ with $t \geq\lambda/2$. We can conclude that $\bx^*$ is a global minimizer of \eqref{L0O}. In fact, $\nabla f(\bx^*)=(0~-4~-4)^\top$ and  $\bx^*-\tau \nabla f(\bx^*)=(t~4\tau~4\tau)^\top$. This and \eqref{tau-sta-cond} show that $\bx^*$ is a $\tau$-stationary point for some $\tau\in(1,\lambda/8]$ due to
\begin{eqnarray*}
\nabla_1 f(\bx^*) = 0~{\rm and}~|x_1|= t\geq\lambda/2=\sqrt{2\lambda \lambda/8}\geq\sqrt{2\lambda\tau},\\
|\nabla_2 f(\bx^*)|=|\nabla_3 f(\bx^*)| =4=\sqrt{2\times8}\leq\sqrt{2\lambda/\tau}.
\end{eqnarray*}
Then it follows from \Cref{relation1} 2) and $\tau>1=1/\ell $ that $\bx^*$ is a unique global minimizer of the problem \eqref{L0O}. Moreover, \Cref{relation1} 1) concludes that a global minimizer (which is $\bx^*$) is also a  $\tau_1$-stationary point with $\tau_1\in(0,1/L)=(0,1/2)$. This is not conflicted with $\bx^*$  being a  $\tau$-stationary point with some $\tau\in(1,\lambda/8]$.

\subsection{Stationary Equation}
\noindent To well express the solution of  \eqref{tau-eq}, define
\begin{eqnarray}\label{T}
 T:=T_{\tau}(\bx,\lambda) &:=& \{i\in \N_n:|x_{i}-\tau \nabla_{i}  f (\bx)|\geq\sqrt{2\tau\lambda} \}.
 \end{eqnarray}
Based on above set, we introduce the following \textit{stationary equation}
\begin{eqnarray}\label{station}
F_{\tau}(\bx;T):=
 \left[\begin{array}{c}
                          \nabla_{ T}  f (\bx)\\
                          \bx_{\overline{ T }}\\
                         \end{array}\right]=
                         0.
\end{eqnarray}
 The relationship between \eqref{tau-eq}  and \eqref{station} is revealed by the following theorem.
\begin{theorem}\label{tau-F} For any $\bx \in \R^n $, by letting $\bz:=\bx-\tau \nabla  f (\bx)$,
we have
$$\bx=\P\left(\bz\right) ~~~\Longrightarrow~~~ F_{\tau}(\bx;T)=0 ~~~\Longrightarrow~~~ \bx\in\P\left(\bz\right).$$
\end{theorem}
 \begin{proof} If we have $\bx=\P\left(\bz\right)$, namely, $\P\left(\bz\right)$ is a singleton, then there is no index $i\in T$ such that $|z_i|=\sqrt{2\tau\lambda}$ by \eqref{tau-sta-exp}. This and \eqref{tau-sta-exp} give  rise to $(\P\left(\bz\right))_{T}=\bz_T$. As  a consequence,
 $$0=\bx-\P(\bz)\overset{\eqref{tau-sta-exp}}{=} \left[\begin{array}{c}
\bx_{ T}\\
\bx_{ \overline{T}}
\end{array}\right]
-\left[\begin{array}{c}
\bz_{ T}\\
0
\end{array}\right]
=
\left[\begin{array}{c}
\tau  \nabla_{ T}  f (\bx)  \\
\bx_{ \overline{T}}
\end{array}\right],$$
which suffices to $F_{\tau}(\bx;T)=0$. We now prove the second claim.
For any $ i\in T$, we have $\nabla_{ i}  f (\bx) =0$ from  \eqref{station} and thus $|x_i|\geq \sqrt{2\tau\lambda}$
  from \eqref{T}. For any $i\in \overline T$,
 we have $x_i=0$ from  \eqref{station}  and
 $|\tau \nabla_{i}  f (\bx)|=|x_{i}-\tau \nabla_{i}  f (\bx)|<\sqrt{2\tau\lambda}$  from \eqref{T}. Those together with \Cref{sts} claim the conclusion immediately.\qed
\end{proof}
\begin{remark}\label{lambda-upper-lower}  Note that if $\nabla f(0)=0$, then $0$ is a \ts\ of  the problem \eqref{L0O}, and even a global minimizer if $f$ is convex.  This case is trivial. However, we are more interested in the non-trivial case. Therefore from now on, we always suppose $\nabla f(0)\neq0$ and  denote \begin{eqnarray}
\label{lamda_upper}\underline\lambda:=\min_i\left\{\frac{\tau}{2}|\nabla_i f(0)|^2:\nabla_i f(0)\neq0\right\},~~~~\overline\lambda:=\max_i \frac{\tau}{2}|\nabla_i f(0)|^2.
\end{eqnarray}  
One can check that if $\lambda$ is chosen to satisfy $0<\lambda\leq\underline \lambda$, then $|0-\tau \nabla_{i}  f (0)|\geq\sqrt{2\tau\lambda}$ for any $i\in J:=\{i\in\N_n:\nabla_i f(0)\neq0\}$, which results in $T_{\tau}(\bx,\lambda)=J$ in \eqref{T} and consequently, $F_{\tau}(0;J)\neq0$ due to $\nabla_J f(0)\neq0$. Namely, $0$ is not a \ts\ of the problem \eqref{L0O}. Hence,  the trivial solution $0$ is excluded. 

On the other hand,  
if $\lambda$ is chosen to satisfy $ \lambda>\overline \lambda$, then $T_{\tau}(\bx,\lambda)=\emptyset$ in \eqref{T}. Because of this $F_{\tau}(0;\emptyset)=0$, namely, $0$ is a \ts\ of the problem \eqref{L0O}. Therefore, when it comes to numerical experiments, this $ \overline \lambda$ provides an upper bound to set up a proper $\lambda$.\end{remark}
\section{Newton Method}\label{Sect:Newton-Method-BLS}
\noindent \cref{tau-F} states that a point satisfying the stationary equation is a stronger condition than being a \ts. The advantage of this equation allows us to design an efficient Newton-type algorithm based on its simple form. Based on the stationary equation \eqref{station}, this section casts a Newton-type method. 

\subsection{Algorithm Design}
\noindent To find a solution to the equation \eqref{station}, we first need to locate the index set $T$ which is unknown in general and then solve the equation. Therefore, we employ an adaptively updating rule as follows. For a computed point $\bx^k$, we first calculate an approximation $T_{k}$.
Then with such a fixed set $T_k$, we apply the Newton method on $F_{\tau}(\bx;T_k)$ once into obtaining a direction $\bd^k$. That is, $\bd^k$ is a solution to the following equation system
\begin{equation} \label{snewton-equ}\nabla F_{\tau}(\bx^{k};T_{k})\bd=-F_{\tau}(\bx^{k};T_{k}).\end{equation}
The explicit formula of $F_{\tau}(\bx^{k};T_{k})$ from \eqref{station} implies that  $\bd^k$ satisfies
 \begin{eqnarray}\label{sequation-k-0}
\nabla_{T_{k}}^{2} f (\bx^k)  \bd^k_{T_{k}} &=& \nabla_{   T_{k}, \overline T_{k} }^{2} f (\bx^k)\bx^{k}_{\overline T_{k}}-\nabla_{ T_{k}}f(\bx^k), \\
\bd^{k}_{\overline{ T }_{k}}&=&-\bx^{k}_{\overline{ T }_{k}}.\nonumber
\end{eqnarray}
Now let us take a look at the above formulas. The second part of $\bd^k$ can be derived directly without any difficulties.  To find $\bd^k$,  one needs to solve a linear equation with $|T_k|$ equations and $|T_k|$ variables. If a full Newton direction is taken, then next iterate $\bx^{k+1}=\bx^{k}+\bd^{k}=[(
\bx^{k}_{T_{k}}+ \bd^k_{T_{k}})^\top~0]^\top.$
This means the support set of $\bx^{k+1}$ will be located within  $T_k$. Namely,
\begin{eqnarray}\label{supp-T}
\supp(\bx^{k+1})\subseteq T_k.
\end{eqnarray}
Based on this idea, we modify the standard rule associated with Amijio line search $\bx^{k+1}=\bx^{k}+\alpha\bd^{k}$ as  $\bx^{k+1}=\bx^{k}(\alpha)$, where
 \begin{eqnarray}\label{xk-alpha}
\bx^{k}(\alpha):=\left[\begin{array}{cc}
\bx^{k}_{T_{k}}+\alpha \bd^k_{T_{k}}\\
\bx^{k}_{\overline T_{k}}+\bd^k_{\overline T_{k}}
\end{array}\right]=\left[\begin{array}{cc}
\bx^{k}_{T_{k}}+\alpha \bd^k_{T_{k}}\\
0
\end{array}\right].
\end{eqnarray}
For notational convenience, let 
 \begin{eqnarray}\label{HTT}
 J_{k} &:=& T_{k-1}\backslash T_{k},~~~S_k:=\widetilde T_{k}\backslash T_{k-1},\\
 g^k&:=&\nabla f(\bx^k),~~~~H_k:=\nabla_{{T_{k}}}^{2} f (\bx^k),~~~~
   G_k :=\nabla_{{T_{k}},{J_k}}^{2} f (\bx^k).
\end{eqnarray}
We summarize the framework of the algorithm in  \Cref{framework-SNSCO}. 
 \begin{algorithm}[H]
\caption{Newton-type method for the $\ell_0$-regularized optimization (\nzno)}
\label{framework-SNSCO}
\begin{algorithmic}
\STATE{If $\nabla f(0)=0$, then return the solution $0$ and terminate the algorithm. Otherwise, perform the following steps. Give parameters $\tau>0,\delta>0, \lambda\in(0,\underline \lambda), \sigma\in(0,1/2),\beta\in(0,1)$. Initialize $\bx^0$. Set $T_{-1}=\emptyset$ and $k\Leftarrow0$ }
\WHILE{The halting conditions are violated}
\STATE{{\bf Step 1.} Set  $T_k= \widetilde T_k$  if $S_k\neq\emptyset$, and $T_k=T_{k-1}$ otherwise, where $\widetilde T_k$ is computed by\\
 \parbox{.92\textwidth}{\begin{eqnarray}\label{Tk}
\widetilde T_k= \{i\in \N_n:|x_{i}^k-\tau g^k_{i} |\geq\sqrt{2\tau\lambda} \}.
\end{eqnarray}}}
\STATE{{\bf Step 2.} If \eqref{sequation-k-0} is solvable and its solution $\bd^k$ satisfies\\
\parbox{.92\textwidth}{\begin{eqnarray}\label{condition}
\langle g^k_{T_{k} } ,  \bd^{k}_{T_{k} }\rangle\leq-\delta \|\bd^{k}\|^2+ ({1}/{4\tau})\|\bx^{k}_{\overline T_{k} }\|^2,
\end{eqnarray}}\\
\hspace{13mm} then update  $\bd^k$ by  solving \eqref{sequation-k-0}, namely by Newton~direction,\\
\parbox{.92\textwidth}{\begin{eqnarray}\label{d-k-nonsingular}
H_k \bd^k_{T_{k}} = G_k  \bx^{k}_{ J_{k} }-g^k_{ T_{k}},\hspace{1cm}
\bd^{k}_{\overline{ T }_{k}}=-\bx^{k}_{\overline{ T }_{k}}.
\end{eqnarray}}\\
\hspace{13mm} Otherwise,  update  $\bd^k$ by  Gradient~direction\\
 \parbox{.92\textwidth}{\begin{eqnarray}\label{d-k-singular}
\hspace{.3cm} \bd^k_{T_{k}}  = -g^k_{ T_{k}},\hspace{1cm}
\bd^{k}_{\overline{ T }_{k}}=-\bx^{k}_{\overline{ T }_{k}}.
\end{eqnarray}}
}
\STATE{{\bf Step 3.} Find the smallest non-negative integer $m_k$ such that\\
\parbox{0.8\textwidth}{\begin{eqnarray}\label{armijio}
f(\bx^{k}({\beta}^{m_k}))\leq f(\bx^k)+\sigma\beta^{m_k} \langle g^k,  \bd^{k} \rangle.
\end{eqnarray}}
}
\STATE{{\bf Step 4.}  Set $\alpha_k=\beta^{m_k}$, $\bx^{k+1}=\bx^{k}( \alpha_k)$ and $k\Leftarrow k+1$.}
\ENDWHILE
\RETURN $\bx^k$
\end{algorithmic}
\end{algorithm}
\noindent From  \Cref{framework-SNSCO}, the following facts are easy to be achieved:
\begin{eqnarray}\label{clarify}~~~~~~~~~~\left\{
\begin{array}{l}
-\bd^{k}_{\overline{ T }_{k}}=\bx^{k}_{\overline{ T }_{k}}=\left[
\begin{array}{c}
\bx^{k}_{T_{k-1}\cap \overline{ T }_{k}}\\
0
\end{array}
\right]=\left[
\begin{array}{c}
\bx^{k}_{T_{k-1} \setminus T _{k}  }\\
0
\end{array}
\right]\overset{\eqref{HTT}}{=}\left[
\begin{array}{c}
\bx^{k}_{J_k}\\
0
\end{array}
\right],\\
\nabla^2_{T_{k}\cup J_{k}  } f(\bx^k)=
\left[
 \begin{array}{c c}
   H_k & G_k  \\
  G_k ^\top& \nabla^2_{  J_{k} }f(\bx^k)\\
 \end{array}
\right].
\end{array} \right.
\end{eqnarray}
We emphasize that $J_k$ captures all nonzero elements in
$\bx^{k}_{\overline{ T }_{k}}$. This and \eqref{clarify} also allow  us to explain that \eqref{sequation-k-0} is rewritten as \eqref{d-k-nonsingular}. Therefore, we will see more $J_k$ instead of  $\overline{ T }_{k}$ being used in convergence analysis.
\begin{lemma}\label{lemma-0} If  $\bd^k$ is from $\eqref{d-k-nonsingular}$, then we have
\begin{eqnarray}
\label{bounded-H} \hspace{1.5cm}\langle g^k_{T_{k} }  ,  \bd^{k}_{T_{k} }\rangle + \langle \bd^{k}_{ T_{k}}, H_k \bd^{k}_{  T_{k}}\rangle =- \langle \bd^{k}_{T_{k}\cup J_{k} },  \nabla^2_{T_{k}\cup J_{k}  } f(\bx^k)\bd^{k}_{T_{k}\cup J_{k} } \rangle+
\langle \bd^{k}_{ J_{k} },   \nabla^2_{ J_{k} }f(\bx^k )\bd^{k}_{ J_{k} }\rangle. 
\end{eqnarray}
\end{lemma}
\begin{proof}
If  $\bd^k$ is from \eqref{d-k-nonsingular}, then we have the following chain of equations,
       \begin{eqnarray*}
       &&\langle \bd^{k}_{T_{k}\cup J_{k} },  \nabla^2_{T_{k}\cup J_{k}  } f(\bx^k)\bd^{k}_{T_{k}\cup J_{k} } \rangle\\
&\overset{\eqref{clarify}}{=}&
\left[
 \begin{array}{cc}
   \bd^{k}_{T_{k}  }  \\
  \bd^{k}_{ J_{k} }
 \end{array}
\right]^\top \left[
 \begin{array}{c}
   H_k\bd^{k}_{T_{k} }+ G_k \bd^{k}_{  J_{k} } \\
   G_k ^\top\bd^{k}_{T_{k} }+\nabla^2_{ J_{k}  }f(\bx^k)\bd^{k}_{  J_{k} }
    \end{array}
\right] \nonumber\\
&\overset{\eqref{clarify}}{=}& \langle  \bd^{k}_{T_{k} }, H_k\bd^{k}_{T_{k} }- G_k \bx^{k}_{  J_{k} } \rangle - \langle \bx^{k}_{  J_{k} },  G_k^\top\bd^{k}_{  T_{k}}\rangle+
\langle \bd^{k}_{ J_{k} },   \nabla^2_{ J_{k} }f(\bx^k )\bd^{k}_{ J_{k} }\rangle\\
&=& 2\langle  \bd^{k}_{T_{k} }, H_k\bd^{k}_{T_{k} }- G_k \bx^{k}_{  J_{k} } \rangle - \langle H_k\bd^{k}_{T_{k} },   \bd^{k}_{  T_{k}}\rangle+
\langle \bd^{k}_{ J_{k} },   \nabla^2_{ J_{k} }f(\bx^k )\bd^{k}_{ J_{k} }\rangle\\
  &\overset{\eqref{d-k-nonsingular}}{=}&-2\langle g^k_{T_{k} }  ,  \bd^{k}_{T_{k} }\rangle-
 \langle \bd^{k}_{ T_{k}}, H_k \bd^{k}_{  T_{k}}\rangle+
\langle \bd^{k}_{ J_{k} },   \nabla^2_{ J_{k} }f(\bx^k )\bd^{k}_{ J_{k} }\rangle,\nonumber
\end{eqnarray*}
which concludes our claim immediately.\qed
\end{proof}
\Cref{lemma-0} indicates that  if  $\nabla^2_{T_{k}\cup J_{k} } f(\bx^k) $ has a positive lower and upper bound, so is $H_k$   bounded from below and $\nabla^2_{ J_{k} }f(\bx^k )$  bounded from above, then \eqref{condition} is satisfied in each step under some properly chosen $\delta$ and $\tau$. This allows the Newton direction to be always imposed. Apparently, $\nabla^2_{T_{k}\cup J_{k} } f(\bx^k) $ being bounded from below can be guaranteed by some assumptions, such as the strong  convexity of $f$, which, however, is a strong assumption. To overcome this, the gradient direction compensates the case when the condition \eqref{condition} is violated.


\subsection{Global and quadratic convergence} 
\noindent As mentioned in \cref{lambda-upper-lower}, if $\nabla f(0)=0$, then $0$ is a \ts\ of  the problem \eqref{L0O}, and even a global minimizer if $f$ is convex.  But this case is trivial.  Therefore, we focus on the case of $\nabla f(0)\neq0$ in \Cref{framework-SNSCO}.  Before our main results, we  define some parameters by
\allowdisplaybreaks\begin{eqnarray}\label{alpha-r}
~~~~\overline{\alpha}&:=&\min\left\{\frac{ 1-2\sigma }{L/\delta -\sigma}
 ,\frac{2(1-\sigma)\delta}{L},~~1\right\},\nonumber\\
\overline{\tau}&:=& \min\left\{\frac{ 2\overline{\alpha} \delta  \beta}{ nL^2},
~~  \frac{{\overline{\alpha}\beta}}{n} ,~~ \frac{1}{4L}\right\}, \\
\rho&:= &\min\left\{\frac{ 2\delta -  n\tau L^2}{2},~~\frac{2-  n\tau}{2}  \right\}.\nonumber
\end{eqnarray}
Our first result shows that the direction in each step of \nzno\ is a descent one with a decent declining rate,  no matter it is taken from the Newton or the gradient direction.
\begin{lemma}[Descent property] \label{descent-direction}
Let $f$ be  strongly smooth with $L>0$ and $\overline{\tau}, \rho$ be defined as \eqref{alpha-r}. Then for any $\tau\in(0,\overline{\tau})$, it holds $\rho>0$ and
\begin{eqnarray}\label{decreasing-direction}
 \langle g^k, \bd^{k}\rangle  \leq -  \rho\| \bd^{k}\|^2-\frac{\tau}{2}  \| g^k_{T_{k-1}} \|^2.
 \end{eqnarray}
\end{lemma}
\begin{proof}It follows from \eqref{alpha-r}
that  $\overline{\alpha}\leq1$ and thus $\overline{\alpha} \beta <1$ due to $\beta\in(0,1)$.  Hence
 $\overline{\tau}\leq \min\left\{ 2\delta/(nL^2),2/n\right\},$ which immediately shows $\rho>0$ if $\tau\in(0,\overline{\tau})$.
In addition, if $\bd^k$ is updated by  \eqref{d-k-nonsingular}, then\begin{eqnarray}\label{nabla-T-f-0}
\|g^k_{T_{k}}\|
 \overset{\eqref{d-k-nonsingular}}{=} 
\|H_k  \bd^k_{T_{k}}-G_k  \bx^{k}_{ J_{k} } \| 
 \overset{\eqref{clarify}}{=}    \|[H_k~G_k ] \bd^{k}_{T_{k}\cup J_{k} } \|
\overset{\eqref{clarify}}{\leq}L\| \bd^k \|,
 \end{eqnarray}
where the inequality holds because of
$\|[H_k~G_k ]\|_2\leq \|\nabla^2_{T_{k}\cup J_{k} } f(\bx^k)\|_2\leq L$
due to strong smoothness of $f$ with the constant $L$. We now prove the conclusion by two cases. 

\textbf{Case i)} {$S_k= \emptyset$}.  Step 1 in  \Cref{framework-SNSCO} sets $T_{k} = T_{k-1}$. Consequently, $J_k=T_{k-1}\setminus T_{k}=\emptyset$ and  $\bd^{k}_{\overline T_{k} } =-\bx^{k}_{\overline T_{k} } =0$ from \eqref{clarify}. If $\bd^k$ is updated by  \eqref{d-k-nonsingular}, then it holds
\begin{eqnarray} \hspace{5mm} 2\langle g^{k}, \bd^{k} \rangle
&=&  2\langle g^{k}_{T_{k} },  \bd^{k}_{T_{k} }\rangle- 2\langle g^{k}_{\overline T_{k} } ,  \bx^{k}_{\overline T_{k} }\rangle=  2\langle g^{k}_{T_{k} },
\bd^{k}_{T_{k} }\rangle\nonumber\\
&\overset{\eqref{condition}}{\leq}&-2\delta \|\bd^{k} \|^2+ \|\bx^{k}_{\overline T_{k} }\|^2/(2\tau)=-2\delta \|\bd^{k} \|^2\nonumber\\
&\leq&-2\delta \|\bd^{k} \|^2+n\tau \| g^k_{T_{k}} \|^2- \tau   \| g^k_{T_{k}} \|^2\nonumber\\
&\overset{\eqref{nabla-T-f-0}}{\leq}&-[2\delta-\tau L^2] \|\bd^{k} \|^2- \tau   \| g^k_{T_{k}} \|^2\nonumber\\
\label{fd-TT-0}&\overset{\eqref{alpha-r}}{\leq}&-2\rho \| \bd^k\|^2-\tau  \|  g^k_{T_{k-1}} \|^2,
\end{eqnarray}
where the last inequality holds due to $T_{k} = T_{k-1}$. If $\bd^k$ is updated by  \eqref{d-k-singular}, then it follows from $\bd^{k}_{T_{k} }=-g^k_{T_{k}}=-g^k_{T_{k-1}}$ that
\begin{eqnarray}2\langle g^{k}, \bd^{k} \rangle
&=&   2\langle g^{k}_{T_{k} },  \bd^{k}_{T_{k} }\rangle- 2\langle g^{k}_{\overline T_{k} } ,  \bx^{k}_{\overline T_{k} }\rangle=  -2  \|\bd^{k}_{T_{k}} \|^2\nonumber\\
&\leq& - 2  \|\bd^{k}_{T_{k}} \|^2+n\tau   \|\bd^{k}_{T_{k}} \|^2-\tau \|\bd^{k}_{T_{k}} \|^2\nonumber\\
&=& -(2-n\tau)  \|\bd^{k}\|^2-\tau \|g^k_{T_{k-1}} \|^2\nonumber\\
&\overset{\eqref{alpha-r}}{\leq}&  -2\rho \| \bd^k\|^2-\tau  \|  g^k_{T_{k-1}} \|^2.
\end{eqnarray}

\textbf{Case ii)} $S_{k} \neq \emptyset$. For any $i\in S_k=\widetilde  T_{k} \setminus T_{k-1}= T_{k} \setminus T_{k-1}$,  we have  $x_i^k=0$
because of $\supp(\bx^{k}) \subseteq T_{k-1}$ by \eqref{supp-T}.
Then the definition of $T_k=\widetilde  T_{k}$ in \eqref{Tk} gives rise to
 \begin{eqnarray}\label{existence-alpha-facts-21}
\forall i\in S_k,~~  |\tau g^k_i |^2=|x_i^k-\tau g^k_i|^2 \geq 2\tau\lambda  >  |x_j^k-\tau g^k_j|^2,~~ \forall j\in  J_{k}.
 \end{eqnarray}
This   suffices to the following chain of inequalities
\begin{eqnarray*}
&&(|J_k|/|S_k|) \tau^2\left[  \| g^k_{T_{k}}  \|^2 -\| g^k_{T_{k} \cap T_{k-1}}  \|^2 \right]\\
&=&(|J_k|/|S_k|)\tau^2 \| g^k_{S_k} \|^2\\
& \overset{\eqref{existence-alpha-facts-21}}{\geq} & |J_k|  2\tau\lambda    \overset{\eqref{existence-alpha-facts-21}} {>}  
 \|\bx^k_{J_{k}}-\tau g^k_{J_{k}}  \|^2\\
 &=&  \|\bx^k_{J_{k}}\|^2-2\tau  \langle\bx^k_{J_{k}},g^k_{J_{k}}  \rangle+ \tau^2\| g^k_{J_{k}} \|^2 \\
 & \overset{\eqref{clarify}}{=}&   \|\bx^k_{\overline T_{k}}\|^2-2\tau  \langle\bx^k_{J_{k}},g^k_{J_{k}} \rangle+ \tau^2\| g^k_{J_{k}}  \|^2\\
 &=&  \|\bx^k_{\overline T_{k}}\|^2-2\tau  \langle\bx^k_{J_{k}},g^k_{J_{k}} \rangle+ \tau^2\left[ \| g^k_{T_{k-1}}  \|^2 -~\| g^k_{T_{k} \cap T_{k-1}}  \|^2 \right]
 \end{eqnarray*}
Since $|J_k|/|S_k|\leq n$, the above inequalities result  in our first fact
 \begin{eqnarray}
\label{existence-alpha-facts-3} - 2\langle \bx^k_{J_{k}},g^k_{J_{k}} \rangle &\leq& n\tau  \| g^k_{T_{k}} \|^2-\tau  \| g^k_{T_{k-1}} \|^2-  \|\bx^k_{\overline T_{k}}\|^2/\tau\\
\label{xk-T1-T1} &\overset{\eqref{nabla-T-f-0}}{\leq}&   n\tau  L^2 \| \bd^k\|^2-\tau  \| g^k_{T_{k-1}}  \|^2-  \|\bx^k_{\overline T_{k}}\|^2/\tau.
 \end{eqnarray}
Now we are ready to establish our claim.  If $\bd^k$ is updated by  \eqref{d-k-nonsingular}, then
\begin{eqnarray}\label{fd-TT} 2\langle  g^k_{T_{k} }  ,
\bd^{k}_{T_{k} }\rangle \overset{\eqref{condition}}{\leq} -2\delta \|\bd^{k} \|^2+ \|\bx^{k}_{\overline T_{k} }\|^2/(2\tau).
\end{eqnarray}
The direct calculation yields the following chain of inequalities,
\begin{eqnarray*}
   2\langle g^{k} , \bd^{k} \rangle
&=&  2\langle g^k_{T_{k} }  ,  \bd^{k}_{T_{k} }\rangle- 2\langle g^{k}_{\overline T_{k} }  ,  \bx^{k}_{\overline T_{k} }\rangle
 \overset{\eqref{clarify}}{=}   2 \langle g^{k}_{T_{k} }  ,  \bd^{k}_{T_{k} }\rangle- 2\langle g^{k}_{J_{k} }  ,  \bx^{k}_{ J_{k} }\rangle\\
&\overset{\eqref{fd-TT}, \eqref{xk-T1-T1}}{\leq}&{ {- ( 2\delta - n\tau L^2  ) }} \| \bd^k \|^2-\|\bx^k_{\overline T_{k} }\|^2/(2\tau) -\tau  \| g^{k}_{T_{k-1}} \|^2\\
&\overset{\eqref{alpha-r}}{\leq}& -2\rho \| \bd^k\|^2-\tau  \| g^{k}_{T_{k-1}}  \|^2.
 \end{eqnarray*}
If $\bd^k$ is updated by  \eqref{d-k-singular}, then $\bd^{k}_{T_{k} }=-g^k_{T_{k} }$ yields that
 \begin{eqnarray*}
 2\langle  g^{k}, \bd^{k} \rangle&=& 2\langle g^{k}_{T_{k} }  ,  \bd^{k}_{T_{k} }\rangle- 2\langle g^{k}_{\overline T_{k} }  ,  \bx^{k}_{\overline T_{k} }\rangle= -2 \| \bd^k_{T_{k}}\|^2  - 2\langle g^{k}_{J_{k} }  ,  \bx^{k}_{J_{k} }\rangle\\
&\overset{\eqref{existence-alpha-facts-3}}{\leq}& -2 \| \bd^k_{T_{k}}\|^2  + n\tau\|g^{k}_{T_{k} }  \|^2-\|\bx^k_{\overline T_{k} }\|^2/ \tau  -\tau  \| g^{k}_{T_{k-1}}  \|^2\\
&\overset{\eqref{clarify}}{=}& { {-(2-n\tau)}} \| \bd^k_{T_{k}}\|^2- \|\bd^k_{\overline T_{k} }\|^2/ \tau  -\tau  \| g^{k}_{T_{k-1}}  \|^2\\
&\leq & {{-(2- n\tau)}} (\| \bd^k_{T_{k}}\|^2+ \|\bd^k_{\overline T_{k} }\|^2)  -\tau  \| g^{k}_{T_{k-1}} \|^2\\
&\overset{\eqref{alpha-r}}{\leq}& -2\rho \| \bd^k\|^2-\tau  \|g^{k}_{T_{k-1}}\|^2,
 \end{eqnarray*}
where the second inequality is from $-1/\tau\leq\tau-2\leq n\tau-2$ for any $\tau>0$.\qed
\end{proof}
 Our next result shows that $\alpha_k$ exists and is bound away from zero. This means the step length to update next point is well defined and would not be too small, which is expected to  {speed up the convergence.}
\begin{lemma} [Existence and boundedness of $\alpha_k$] \label{existence-alpha}
Let $f$ be  strongly smooth with $L>0$ and $\overline{\alpha},\overline{\tau}$ be defined as \eqref{alpha-r}. Then
  \begin{eqnarray}\label{alpha-decreasing-property}
 f(\bx^{k}(\alpha))\leq f(\bx^k)+\sigma \alpha \langle g^{k},  \bd^{k} \rangle
 \end{eqnarray}
holds for any $k\geq 0$ and any parameters \begin{eqnarray*}\label{decreasing-property}
 0<\alpha\leq\overline{\alpha},~~~0<\delta\leq\min \{1,2L\},~~~
 { {0<\tau\leq\min\left\{  \alpha \delta /(nL^2),~\alpha/n,~ 1/({4L}) \right\}}}.\end{eqnarray*}
Moreover, for any $\tau\in(0,\overline{\tau})$,  we have $\inf_{k\geq 0} \{\alpha_k\} \geq \beta \overline \alpha >0.$
\end{lemma}
\begin{proof} If $0<\alpha\leq\overline{\alpha}$ and $0<\delta\leq\min \{1,2L\}$, we have
 \begin{eqnarray}
\label{alpha-L}
\alpha\leq \frac{2(1-\sigma)\delta}{L},~~~~\alpha \leq \frac{ 1-2\sigma }{{L}/ \delta  -\sigma}  \leq \frac{ 1-2\sigma }{\max\{0, {L}  -\sigma\}}. \end{eqnarray}
Since $f$ is  strongly smooth, we obtain that
 \begin{eqnarray*}
\label{2f-2f}  &&2f(\bx^{k}(\alpha))- 2f(\bx^{k}) -  2\alpha\sigma\langle g^{k}, \bd^{k}\rangle  \\
 &\overset{(\ref{strong-smooth})}{\leq}&
 2\langle g^{k}, \bx^{k}(\alpha)-  \bx^{k}\rangle+  L  \|\bx^{k}(\alpha)-  \bx^{k}\|^2-  2\alpha\sigma\langle g^{k}, \bd^{k}\rangle \nonumber\\
 &\overset{(\ref{xk-alpha})}{=}&  \alpha(1-\sigma)2\langle g^{k}_{T_{k} } ,  \bd^{k}_{T_{k} }\rangle-(1-\alpha\sigma)2\langle g^{k}_{\overline T_{k} }  ,  \bx^{k}_{\overline T_{k} }\rangle+  L \left[\alpha^2\|\bd^{k}_{T_{k} }\|^2+ \|\bx^{k}_{\overline T_{k} }\|^2\right]\nonumber\\
 &\overset{(\ref{clarify})}{=} &  \alpha(1-\sigma)2\langle g^k_{T_{k} }  ,  \bd^{k}_{T_{k} }\rangle-(1-\alpha\sigma)2\langle g^{k}_{ J_{k} },  \bx^{k}_{J_{k}}\rangle +L \left[\alpha^2\|\bd^{k}_{T_{k} }\|^2+ \|\bx^{k}_{\overline T_{k} }\|^2\right]=:\psi.\nonumber
 \end{eqnarray*}
To prove \eqref{alpha-decreasing-property}, one  needs to show $\psi\leq0$. Similar to the proof of \cref{descent-direction}, we consider two cases. \textbf{Case i)}  {$S_k= \emptyset$}. Step 1 in  \Cref{framework-SNSCO} sets $T_{k} = T_{k-1}$, and thus $J_k=T_{k-1}\setminus T_{k}=\emptyset$. Then we obtain
\begin{eqnarray}
\psi&=&\alpha(1-\sigma)2\langle g^k_{T_{k} }  ,  \bd^{k}_{T_{k} }\rangle  +L  \alpha^2\|\bd^{k}_{T_{k} }\|^2\nonumber\\
&&
\begin{cases}
\overset{(\ref{condition} )}{\leq} -2\alpha(1-\sigma) \delta \|\bd^{k}  \|^2  + L  \alpha^2\|\bd^{k}_{T_{k} }\|^2 ,&\text{if}~\bd^k~ \text{is from}~ (\ref{d-k-nonsingular})\\
 \overset{(\ref{d-k-singular} )}{=} -2\alpha(1-\sigma) \|\bd^{k}_{T_{k} }\|^2  +L  \alpha^2\|\bd^{k}_{T_{k} }\|^2,&\text{if}~\bd^k~ \text{is from}~ (\ref{d-k-singular})
\end{cases}
\nonumber \\
&\leq& -2\alpha(1-\sigma) \delta \|\bd^{k}  \|^2  + L  \alpha^2\|\bd^{k}\|^2\nonumber \\
\label{rem-tua-delta}&=&\alpha(L  \alpha-2(1-\sigma) \delta) \|\bd^{k}  \|^2\overset{(\ref{alpha-L} )}{\leq} 0,
 \end{eqnarray}
 where the third inequality is due to $\delta\leq 1$, $\|\bd^{k}  \|^2 =\|\bd^{k} _{T_{k} }\|^2 $.

\textbf{Case ii)}  {$S_k\neq \emptyset$}. If   $\bd^k$ is from (\ref{d-k-nonsingular}), then we have
 \begin{eqnarray*}
\psi&\overset{\eqref{fd-TT}}  {\leq}& \alpha(1-\sigma) \left[- 2\delta \|\bd^{k}  \|^2 +  ({1}/{2\tau})\|\bx^{k}_{\overline T_{k} }\|^2\right]   + L  \alpha^2\|\bd^{k}_{T_{k} }\|^2 \\
&\overset{\eqref{xk-T1-T1}}  {+}& (1-\alpha\sigma)\left[ n\tau L^2 \| \bd^k\|^2-\tau  \| g^k_{T_{k-1}}  \|^2- ({1}/{\tau}) \|\bx^k_{\overline T_{k}}\|^2 \right] +L \|\bx^{k}_{\overline T_{k} }\|^2\\
&\leq&c_1\| \bd^k\|^2  + c_2\|\bx^k_{\overline{ T }_{k}}\|^2-(1-\alpha\sigma)\tau  \| g^k_{T_{k-1}}  \|^2\\
&\leq& c_1\| \bd^k\|^2  + c_2\|\bx^k_{\overline{ T }_{k}}\|^2,
 \end{eqnarray*}
where $1-\alpha\sigma>0$ due to $0<\alpha<1,0<\sigma\leq1/2$ and $c_1$ and $c_2$ are given by
\begin{eqnarray*}c_1&:=& -\alpha(1-\sigma)2\delta +(1-\alpha\sigma){ {n\tau {L}^2}}+ {L}\alpha^2, \\
&\leq& -\alpha(1-\sigma)2\delta +(1-\alpha\sigma)\delta \alpha+ {L}\alpha^2 \hspace{0.6cm} {\rm by}~1-\alpha\sigma >0, \tau\leq   { \alpha \delta  }/{(n{L}^2)}  \\
&=& \alpha \left[( {L}-\sigma\delta )\alpha-(1-2\sigma)\delta \right]\leq0,\hspace{1.05cm} {\rm by}~L-\sigma\delta>0, 1-2\sigma>0, \eqref{alpha-L}\\
c_2&:=& \alpha(1-\sigma)/(2\tau)-(1-\alpha\sigma)/\tau + {L}\\
&\leq&(1-\alpha\sigma)/(2\tau)-(1-\alpha\sigma)/\tau + {L}\hspace{0.8cm}{\rm by}~1-\alpha\sigma>0\\
&\leq&-(1-\alpha\sigma)/(2\tau) + {L}\leq 0. \hspace{2.0cm} {\rm by}~1-\alpha\sigma>0, \tau\leq {1}/{(4{L})}  \end{eqnarray*}
If $\bd^k$ is updated by  (\ref{d-k-singular}), namely $\bd^{k}_{T_{k}}  =-g^k_{T_{k}}$, then
 \begin{eqnarray*}
\psi&\overset{\eqref{d-k-singular}}  {\leq}& -2\alpha(1-\sigma)  \|\bd^{k}_{T_{k}}  \|^2     + L  \alpha^2\|\bd^{k}_{T_{k} }\|^2 \\
&\overset{ \eqref{existence-alpha-facts-3}}  {+}& (1-\alpha\sigma)\left[ n\tau  \| g^k_{T_{k}}  \|^2-\tau  \| g^k_{T_{k-1}}  \|^2- ({1}/{\tau}) \|\bx^k_{\overline T_{k}}\|^2 \right] +L \|\bx^{k}_{\overline T_{k} }\|^2\\
&\overset{ \eqref{d-k-singular}}  {\leq}& c_3\| \bd^k_{T_{k}}\|^2  + c_4\|\bx^k_{\overline{ T }_{k}}\|^2-(1-\alpha\sigma)\tau  \| g^k_{T_{k-1}}  \|^2,
 \end{eqnarray*}
  where $c_3$ and $c_4$ are given by
  \begin{eqnarray*}c_3&:=& -2\alpha(1-\sigma)+(1-\alpha\sigma){ {n\tau }}+ {L}\alpha^2  \\
  &\leq&-2\alpha(1-\sigma)+(1-\alpha\sigma)\alpha+ {L}\alpha^2  \hspace{1.1cm} {\rm by}~1-\alpha\sigma>0,\tau\leq {\alpha}/{n}\\
&=&\alpha\left[( {L}-\sigma )\alpha-(1-2\sigma) \right]\\
&\leq&\alpha\left[\max\{0, {L}-\sigma\}\alpha-(1-2\sigma) \right]\leq0 \hspace{0.6cm} {\rm by}~1-2\sigma>0, \eqref{alpha-L} \\
c_4&:=& -(1-\alpha\sigma)/\tau+ {L}\\
&\leq&-1/(2\tau)+{L}\leq0,\hspace{3.4cm} {\rm by}~1-\alpha\sigma\geq  { 1}/{2},\tau\leq {1}/({4{L}}) 
\end{eqnarray*}
Thus we verify  \eqref{alpha-decreasing-property}. If further $\tau\in(0, \overline \tau)$, then for any $\alpha\in[\beta \overline  \alpha,   \overline \alpha]$, one can check that
 $$0<\tau\overset{ \eqref{alpha-r}}  {<}\min\left\{ \overline{\alpha} \delta  \beta/(n{L}^2),~\overline{\alpha}\beta/n,~1/(4{L}) \right\}\leq\min\left\{ \alpha \delta / (nL ^2),~ \alpha/n ,~1/(4{L})\right\}.$$
 Therefore, (\ref{alpha-decreasing-property}) holds for any for any $\alpha\in[\beta \overline  \alpha,   \overline \alpha]$. 
Finally, the Armijo-type step size rule  means  that  $\{\alpha_k\}$ must be bounded from below by $\beta \overline \alpha$, that is,
\begin{equation} \label{Positive-Lower-Bound}
\inf_{k\geq 0}\{\alpha_k\} \geq \beta \overline \alpha > 0.
\end{equation}
The whole proof is completed.\qed
\end{proof}
\Cref{existence-alpha} allows us to conclude that the objective $f$ is strictly decreasing for each step, and the difference of two consecutive iterates and the entries of the stationary equation will vanish.
\begin{lemma} \label{inf-alpha-xk1-xk}
Let $f$ be  strongly smooth with $L>0$ and $\overline{\tau}$ be defined as $(\ref{alpha-r})$. Let  $ \{\bx^k\}$ be the sequence generated by {\nzno} with $\tau\in(0,\overline{\tau})$
and $\delta \in(0,\min \{1,2{L}\})$. Then  $\{f(\bx^k)\}$ is a strictly nonincreasing sequence and  
\begin{eqnarray}\label{Fk-gk-0}\lim_{k\rightarrow\infty}\max\left\{~\|F_{\tau}(\bx^{k};T_{k})\|,~\|\bx^{k+1}-\bx^k\|,~\| g^k_{T_{k-1}}  \|,~\| g^k_{T_{k}}  \|\right\}=0.
\end{eqnarray}
\end{lemma}
\begin{proof}
By (\ref{alpha-decreasing-property}),  (\ref{decreasing-direction}) and denoting $c_0:=\sigma \overline \alpha \beta\rho,$ we have
\begin{eqnarray*}
f(x^{k+1})-f(x^k) \leq  \sigma \alpha_k \langle g^k,  \bd^{k} \rangle
&\overset{(\ref{decreasing-direction})}{\leq}& - \sigma \alpha_k \rho\| \bd^k\|^2-\frac{\tau}{2}\| g^k_{T_{k-1}} \|^2\\
&\overset{(\ref{Positive-Lower-Bound})}{\leq}&  -c_0 \| \bd^k\|^2-\frac{\tau}{2}  \| g^k_{T_{k-1}}  \|^2.
\end{eqnarray*}
Then it follows from the above inequality that
\begin{eqnarray*}
 \sum^{\infty}_{k=0}\Big[c_0\| \bd^k\|^2+ \frac{\tau}{2}  \| g^k_{T_{k-1}} \|^2\Big]
& \leq&
 \sum^{\infty}_{k=0}\Big[ f(\bx^k)-f(\bx^{k+1})\Big]\\
 &= & \Big[f(\bx^0)-\lim _{k\rightarrow +\infty}f(\bx^k)\Big]
 <+\infty,
\end{eqnarray*}
where the last inequality is due to $f$ being bounded from below.
Hence  $\| \bd^k\|\rightarrow0, \| g^k_{T_{k-1}} \|\rightarrow0$, which
suffices to $\|\bx^{k+1}-\bx^k\|\rightarrow0$ because of
$$ \|\bx^{k+1}-\bx^k\|^2\overset{(\ref{xk-alpha})}{=}  \alpha_k^2 \| \bd^k_{T_{k}}\|^2+\| \bx^k_{\overline T_{k}}\|^2\leq \| \bd^k_{T_{k}}\|^2+\| \bd^k_{\overline T_{k}}\|^2=\| \bd^k\|^2.$$
The above relation also indicates $\| \bx^k_{\overline T_{k}}\|^2\rightarrow0$. In addition, if $ \bd^k$ is taken from (\ref{d-k-nonsingular}), then $\|g^k_{ T_{k}} \|  \leq  L \|\bd^k\|\rightarrow0$ by (\ref{nabla-T-f-0}). If it is taken from  (\ref{d-k-singular}) then $\|g^k_{ T_{k}}  \|=\| \bd^k_{T_{k}}\|\rightarrow0$. Those together with (\ref{station}) that
$\|F_{\tau}(\bx^{k};T_{k})\|^2= \|g^k_{ T_{k}}  \|^2 + \| \bx^k_{\overline T_{k}}\|^2 \rightarrow0,$ finishing the whole proof.\qed
\end{proof}
We are ready to conclude from \Cref{inf-alpha-xk1-xk} that the index set of $T_K$ can be  {identified within finite steps and the sequence converges to a $\tau$-stationary point or a local minimizer globally, which are presented by the following theorem.
\begin{theorem}[Convergence and identification of $T_k$] \label{whole-sequence-converge}
Let $f$ be  strongly smooth with $L>0$ and $\overline{\tau}$ be defined as $(\ref{alpha-r})$.
Let  $ \{\bx^k\}$ be the sequence generated by {\nzno} with $\tau\in(0,\overline{\tau})$ and $\delta \in(0,\min \{1,2L\})$. Then the following results hold.
\begin{enumerate}
\item[{\rm 1)}]For any sufficiently large $k$, $T_{k}\equiv T_{k-1}\equiv :T_\infty$.
\item[{\rm 2)}]Any accumulating point (say $\bx^*$) of the sequence satisfies
\begin{equation}\label{sta-model}
 \nabla_{ T_\infty}  f (\bx^*)=0, ~~~~\bx^*_{\overline{ T }_\infty}=0,~~~~\supp(\bx^*)\subseteq  T_\infty
\end{equation}
and is non-trivial ($\bx^*\neq0$), and it is necessary a $\tau_*$-stationary point of \eqref{L0O} with 
\begin{equation}\label{sta-model-tau*}
0<\tau_*<
\min\Big\{\overline{\tau}, \min_{i\in\supp(\bx^*)}|x_i^*|/(2\lambda))\Big\}.
\end{equation}
\item[{\rm 3)}] If $\bx^*$ is isolated, then the whole sequence converges to $\bx^*$. 
\end{enumerate}
\end{theorem}
\begin{proof} 1) For any sufficiently large $k$, $T_{k}\equiv T_{k-1}$ indicates $S_k=\emptyset$ by Step 1 in \Cref{framework-SNSCO}. Suppose there is a subsequence $\mathcal K$ of $\{0,1,2,\cdots\}$ such that  $S_k\neq \emptyset, k\in\mathcal K$. Then we have $S_k=\widetilde T_{k}\backslash T_{k-1}= T_{k}\backslash T_{k-1}\neq \emptyset, k\in\mathcal K$. \Cref{inf-alpha-xk1-xk} shows that $g^k_{ T_{k}} \rightarrow 0$, which yields $g^k_{ S_{k}} \rightarrow 0$. This contradicts with $|\tau g^k_{ i}| \geq \sqrt{2\tau\lambda}, i\in S_{k}$ by \eqref{existence-alpha-facts-21}. 

2) Let  $\{\bx^{k_t }\}$ be the convergent subsequence of $\{\bx^{k}\}$  that converges to $\bx^*$. 
Since $\bx^{k_t }\rightarrow \bx^*$ and $\|\bx^{k+1}-\bx^k\| \rightarrow 0$ from \Cref{inf-alpha-xk1-xk}, we have $\bx^{k_t+1 }\rightarrow \bx^*$ and thus $\supp(\bx^*)\subseteq\supp(\bx^{k_t+1 })$ for sufficiently large $k_t$. Then it follows from $\supp(\bx^{k_t+1 })\subseteq T_{k_{t}} \equiv T_\infty$ by (\ref{supp-T}) and claim 1) that $\supp(\bx^*)\subseteq\supp(\bx^{k_t+1 })\subseteq T_\infty.$ Moreover,
 \begin{eqnarray}\label{gTf0}
  \nabla_{T_\infty}f(\bx^*)= \nabla_{T_{k_t}} f(\bx^*)=\underset{{k_t }\rightarrow \infty}{\lim} \nabla_{T_{k_t}} f(\bx^{k_t})=\underset{{k_t }\rightarrow \infty}{\lim} g^{k_t }_{ T_{k_t}}  \overset{(\ref{Fk-gk-0})}{= }0.\end{eqnarray}
Overall, \eqref{sta-model} is true. Next, we claim that $\bx^*\neq0$. Suppose $\bx^*=0$. \Cref{framework-SNSCO} runs infinite steps only when  $\nabla   f (0)\neq0$.  Under such a scenario and $\lambda\in(0,\underline\lambda)$,   by $\bx^{k_t}\rightarrow \bx^*=0$, for sufficiently large $k$, there is a sufficiently small $\varepsilon>0 $ such that
\begin{eqnarray}\label{x-g-tau-lam}
|x_{i}^{k_t}-\tau g^{k_t}_{i} |&\geq& \tau | \nabla_{i}  f (0)|- |x_{i}^{k_t}|-\tau | \nabla_{i}  f (0)-\nabla_{i}  f (\bx^{k_t}) |\nonumber\\
&\overset{\eqref{lamda_upper}}{\geq}& \sqrt{2\tau\underline\lambda}-\varepsilon \geq \sqrt{2\tau \lambda}.
\end{eqnarray}
 Thus $\widetilde T_{k_t}\neq \emptyset$. Recall that in claim 1), $S_k=\emptyset$ for any sufficiently large $k$. This implies $\widetilde T_{k_t}\subseteq  T_{k_t-1}\equiv T_{k_t} \equiv T_{\infty}$. However,  by \eqref{gTf0}, $g^{k_t}_{i} \rightarrow 0$ and $x_{i}^{k_t} \rightarrow 0$, contradicting with \eqref{x-g-tau-lam}. Thus, $\bx^*\neq0$. Now by \eqref{sta-model-tau*}, it is easy to check that
\begin{eqnarray*}
T_*:=\supp(\bx^*) &=&\left\{i\in \N_n: x_{i}^*\neq0\right\}\\
&=& \{i\in \N_n:|x_{i}^*-\tau_* \nabla_{i}  f (\bx^*)|\geq\sqrt{2\tau_*\lambda} \}.
 \end{eqnarray*}
This together with $ \nabla_{T_*}  f (\bx^*)=0, \bx^*_{\overline{ T}_*}=0$ from \eqref{sta-model} suffices to 
$F_{\tau_*}(\bx^*;T_*)=0$. Finally, \cref{tau-F} allows us to claim that $ \bx^*$  is   a $\tau_*$-stationary point.

3) The whole sequence converges because of $\bx^*$ being isolated, \cite[Lemma 4.10]{more1983computing} and
 $\|\bx^{k+1}-\bx^k\|\rightarrow0$ from \Cref{inf-alpha-xk1-xk}. \qed \end{proof}

Finally, we would like to see how fast our proposed method \nzno\ converges.
To proceed that, we need the  locally Lipschitz continuity. We say the Hessian of $f$ is locally Lipschitz continuous around $\bx^*$ with a constant $M_*>0$ if  
$$\|\nabla^2f(\bx)-\nabla^2f(\bx')\|_2\leq M_*\| \bx - \bx'\|.$$
for any points $\bx,\bx'$ in the neighbourhood of $\bx$. In addition, we also need that $f$ is locally strongly convex with a constant $\ell_*>0$ around $\bx^*$. As we mentioned before, the constants $M_*$ and $\ell_*$ depend  on the point $\bx^*$. Now we are able to establish the following results.
\begin{theorem}[Global and quadratic convergence] \label{converge-rate}
Let  $ \{\bx^k\}$ be the sequence generated by \nzno\
and $\bx^*$ be one of its accumulating points.  
Suppose $f$ is  strongly smooth with constant $L>0$  and locally strongly convex with $\ell_*>0$ around $\bx^*$. Choose $\tau\in(0,\overline{\tau})$ and
$\delta \in(0,\min \{1,\ell_*\})$. Then the following results hold.
\begin{enumerate}
\item[{\rm 1)}] The whole sequence converges to $\bx^*$, namely, $\bx^*$ is the limit point.
\item[{\rm 2)}] The Newton direction is always accepted for sufficiently large $k$.
\item[{\rm 3)}]  Furthermore, if the Hessian of $f$ is locally Lipschitz continuous around $\bx^*$ with constant $M_*>0$. Then for sufficiently large $k$,
\begin{eqnarray}
\|\bx^{k+1}-\bx^*\|&\leq& M_*/({2\ell_*})\|\bx ^k-\bx ^*\|^2.
\end{eqnarray}
\end{enumerate}
\end{theorem}
\begin{proof}  1) Denote $T_*:=\supp(\bx^*)$.
\Cref{whole-sequence-converge} shows that  $ \nabla_{T_*} f (\bx ^{*})=0$ and $\bx^*\neq0$. Consider a local region $N(\bx^*):=\{\bx\in\R^n:\|\bx-\bx^*\|<\epsilon_*\}$, where $$\epsilon_*:=\min\Big\{\lambda/(2 \|\nabla_{\overline T_*} f(\bx^*)\|), \min_{i\in T_*}|x^*_i|\Big\}.$$ For any $\bx(\neq\bx^*)\in N(\bx^*)$, we have $T_*\subseteq \supp(\bx)$. In fact if there is a $j$ such that $j\in T_*$ but $j\notin \supp(\bx) $, then we derive a contradiction:
 $$\epsilon_* \leq\min_{i\in T_*}|x^*_i|\leq |x_j^*|=|x_j^*- x_j|\leq \|\bx-\bx^*\|_2 <\epsilon_*. $$
As $f$ is locally strongly convex with $\ell_* >0$ around $\bx^*$, for any $\bx(\neq\bx^*)\in N(\bx^*)$, it holds
\begin{eqnarray*}
&&f(\bx)+\lambda\|\bx\|_0-f(\bx^*)-\lambda\|\bx^*\|_0\\
& \geq & \langle \nabla f(\bx^*), \bx-\bx^*\rangle+(\ell_*/2) \|\bx-\bx^*\|^2+\lambda\|\bx\|_0-\lambda\|\bx^*\|_0\\
&= &  \langle \nabla_{\overline T_*} f(\bx^*), \bx _{\overline T_*}\rangle +(\ell_* /2) \|\bx-\bx^*\|^2+\lambda\|\bx\|_0 -\lambda\|\bx^*\|_0=:\phi.
\end{eqnarray*}
where the first equality is owing to $ \nabla_{T_*} f (\bx ^{*})=0$. Clearly, if $T_* = \supp(\bx)$, then $\bx _{\overline T_*}=0, \|\bx\|_0=\|\bx^*\|_0$ and hence  $\phi=(\ell_*/2) \|\bx-\bx^*\|^2>0$. If $T_* \neq(\subseteq) \supp(\bx)$, then $\|\bx\|_0\geq\|\bx^*\|_0+1$ and thus it gives rise to
\begin{eqnarray*} 
\phi& \geq &  
 -\|\nabla_{\overline T_*} f(\bx^*)\|\| \bx _{\overline T_*}\|+(\ell_* /2) \|\bx-\bx^*\|^2+\lambda\\
 & \geq &  
 -\|\nabla_{\overline T_*} f(\bx^*)\|\| \bx-\bx^*\|+(\ell_* /2) \|\bx-\bx^*\|^2+\lambda\\
  & \geq &  
 -\lambda/2+(\ell_* /2) \|\bx-\bx^*\|^2+\lambda>0.
\end{eqnarray*}
Both cases exhibit that $\bx^*$ is a strictly local minimizer of (\ref{L0O}) and is unique in $N(\bx^*)$, namely, $\bx^*$ is isolated local minimizer in $N(\bx^*)$. So  the whole sequence tends to $\bx^*$ by \Cref{whole-sequence-converge} 3).

2) We first verify $H_k$ is nonsingular when $k$ is sufficiently large and
$$ \langle g^k_{T_{k} } ,  \bd^{k}_{T_{k} }\rangle\leq-\delta \|\bd^{k}\|^2+
\|\bx^{k}_{\overline T_{k} }\|^2/(4\tau). $$
Since $f$ is strongly smooth with ${L}$ and locally strongly convex with $\ell_*  $ around $\bx^*$, it follows
\begin{eqnarray}\label{Hk2s}
{\ell_* }\leq \lambda_i(\nabla^2_{T_{k}\cup J_{k}  } f(\bx^k)), \lambda_i(H_{k}),
\lambda_i(\nabla^2_{J_{k} } f(\bx^k))
 \leq L,
\end{eqnarray}
where $\lambda_i(A)$ is the $i$th largest eigenvalue of $A$. Direct verification yields that
\allowdisplaybreaks      \begin{eqnarray*}
2\langle g^k_{T_{k} }  ,  \bd^{k}_{T_{k} }\rangle
   &\overset{(\ref{bounded-H})}{=}& -\langle \bd^{k}_{T_{k}\cup J_{k} },
   \nabla^2_{T_{k}\cup J_{k}  } f(\bx^k)\bd^{k}_{T_{k}\cup J_{k} }\rangle -\langle H_k \bd^{k}_{ T_{k}}, \bd^{k}_{  T_{k}}\rangle+
\langle \bd^{k}_{ J_{k} },   \nabla^2_{ J_{k}  }f(\bx^k )\bd^{k}_{ J_{k} }\rangle \\
 &\overset{(\ref{Hk2s})}{\leq}&- {\ell_* }\left[\|\bd^{k}_{T_{k}\cup J_{k} }\|^2+\|\bd^{k}_{T_{k} }\|^2\right] + L \|\bx^{k}_{\overline T_{k} }\|^2 \nonumber\\
  &=&- {\ell_* }\left[\|\bd^{k}_{T_{k}\cup J_{k} }\|^2+  \|\bd^{k}_{T_{k} }\|^2+  \|\bd^{k}_{ J_{k} }\|^2-  \|\bd^{k}_{ J_{k} }\|^2\right] + L\|\bx^{k}_{\overline T_{k} }\|^2 \nonumber\\
    &=&- 2{\ell_* } \|\bd^{k}_{T_{k}\cup J_{k} }\|^2+\ell \|\bd^{k}_{ J_{k} }\|^2  + L \|\bx^{k}_{\overline T_{k} }\|^2 \\
    & \overset{(\ref{clarify})}{=}&- 2{\ell_* }\|\bd^{k} \|^2 + ({\ell_* }+ L) \|\bx^{k}_{\overline T_{k} }\|^2 \nonumber\\
   & \leq& - 2{\ell_* }\|\bd^{k} \|^2 + 2L \|\bx^{k}_{\overline T_{k} }\|^2 \\
   &\leq&-2\delta \|\bd^k \|^2+  \|\bx^{k}_{\overline T_{k}}\|^2/(2\tau),
\end{eqnarray*}
where   the last inequality  is owing to $\delta\leq{\ell_* }$ and  $\tau<\overline \tau\leq  1/(4{L})$.
This proves that $\bd^k$ from (\ref{d-k-nonsingular}) is always admitted for sufficiently large $k$.

  3) By \Cref{whole-sequence-converge} 2), for sufficiently large $k$, we have \eqref{sta-model}, which suffices  to
\begin{eqnarray}\label{F0-s-1}
\bx^*_{\overline{T}_k}=0, ~~  \nabla_{T_{k}}  f (\bx^*) =0.\end{eqnarray}
For any $0\leq t \leq1$, by letting $\bx(t):=\bx^* + t (\bx^k-\bx^*)$.
the Hessian of $f$ being  locally Lipschitz continuous at $\bx^*$ derives
\begin{eqnarray}
\label{facts-0-4} &&\|\nabla^{2}_{{T_{k} }:} f (\bx^k)-\nabla_{{T_{k} }:}^{2} f (\bx(t)) \|_2 
\leq  M_*\|\bx^k-\bx(t)\|=(1-t)M_*\|\bx^k- \bx^*\|.
\end{eqnarray}
 Moreover, by Taylor expansion, one has
\begin{eqnarray}\label{facts-0-3}
\nabla  f (\bx ^k)-\nabla  f (\bx ^*)=\int_0^1\nabla^{2} f ( \bx(t))(\bx ^k-\bx ^*)dt.
\end{eqnarray}
Now, we have the following chain of inequalities
 \begin{eqnarray}
\|\bx^{k+1}-\bx^*\|^2&=& \| \bx^{k+1}_{ T_{k} }-\bx^*_{ T_{k} }\|^2+\| \bx^{k+1}_{ \overline T_{k} }-\bx^*_{  \overline T_{k} }\|^2 \nonumber\\
&\overset{(\ref{xk-alpha},\ref{F0-s-1})}{=}&\| \bx^{k+1}_{ T_{k} }-\bx^*_{ T_{k} }\|^2\overset{(\ref{xk-alpha})}{=}\| \bx^{k}_{ T_{k} }-\bx^*_{ T_{k} }+\alpha_k \bd^{k}_{ T_{k} }\|^2\nonumber\\
&=&\|(1-\alpha_k) (\bx^{k}_{ T_{k} }-\bx^*_{ T_{k} })+\alpha_k (\bx^{k}_{ T_{k} }-\bx^*_{ T_{k} }+ \bd^{k}_{ T_{k} })\|^2\nonumber\\
\label{facts-0-5-1}&\leq&(1-\alpha_k)\| \bx^{k}_{ T_{k} }-\bx^*_{ T_{k} }\|^2+\alpha_k \|\bx^{k}_{ T_{k} }-\bx^*_{ T_{k} }+\bd^{k}_{ T_{k} }\|^2 \\
\label{facts-0-5}&\overset{(\ref{Positive-Lower-Bound})}{\leq}&(1-\overline{\alpha}\beta )\| \bx^{k} -\bx^* \|^2+\overline{\alpha} \|\bx^{k}_{ T_{k} }-\bx^*_{ T_{k} }+\bd^{k}_{ T_{k} }\|^2,
\end{eqnarray}
where (\ref{facts-0-5-1}) is due to $\|\cdot\|^2$ is a convex function. From 2), $\bd^k$ is always updated by  (\ref{d-k-nonsingular}) for sufficiently large $k$. Therefore, we have
\allowdisplaybreaks \begin{eqnarray}
\ell_*  \|\bx^{k}_{ T_{k} }-\bx^*_{ T_{k} }+\bd^{k}_{ T_{k} }\|
&\overset{(\ref{sequation-k-0})}{=}&\ell_*  \|H_k^{-1}(\nabla^2_{T_{k} \overline T_{k}  } f (\bx ^k)\bx^k_{\overline T_{k} } -g^k_{T_{k} }  )+\bx^{k}_{ T_{k} }-\bx^*_{ T_{k} }\|\nonumber\\
&\leq& \|\nabla^2_{T_{k} : } f (\bx ^k)\bx^k -g^k_{T_{k} }  -H_{k} \bx^*_{ T_{k} } \|\nonumber\\
&\overset{(\ref{F0-s-1})}{=}& \|\nabla^2_{T_{k} : } f (\bx ^k)\bx^k
- g^k_{T_{k}}-\nabla_{{T_{k}}:}^{2} f (\bx^k)\bx^* +\nabla_{T_{k} } f (\bx ^*) \|\nonumber\\
&\overset{(\ref{facts-0-3})}{=}& \|\nabla_{{T_{k} }:}^{2} f (\bx^k)(\bx^{k}-\bx^*)-\int_0^1\nabla_{{T_{k}}{:}}^{2} f ( \bx(t))(\bx ^k-\bx )dt\|\nonumber\\
&=&\|\int_0^1[\nabla_{{T_{k} }:}^{2} f (\bx^k)-\nabla_{{T_{k} }{:}}^{2} f ( \bx(t))](\bx ^k-\bx ^*)dt\|\nonumber\\
&\leq& \int_0^1\|\nabla_{{T_{k} }:}^{2} f (\bx^k)-\nabla_{{T_{k} }{:}}^{2} f ( \bx(t))\|_2\|\bx ^k-\bx ^*\|dt\nonumber\\
&\overset{(\ref{facts-0-4})}{\leq}&M_*\|\bx ^k-\bx ^*\|^2 \int_0^1(1-t)dt\nonumber\\
\label{facts-7} &\leq&  0.5M_*\|\bx ^k-\bx ^*\|^2.
\end{eqnarray}
It follows from $\bd^k_{\overline T_{k}}=-\bx^k_{\overline T_{k}}$ and (\ref{F0-s-1})  that $\|\bx^k+\bd^k-\bx^*\| = \|\bx^k_{T_{k} }+\bd^k_{T_{k} }-\bx^*_{T_{k} }\|$ and thus
\begin{eqnarray}\label{first-fact} \frac{\|\bx^k+\bd^k-\bx^*\|}{\|\bx^k-\bx^*\|} = \frac{\|\bx^k_{T_{k} }+\bd^k_{T_{k} }-\bx^*_{T_{k} }\|}{\|\bx^k-\bx^*\|}
 \overset{( \ref{facts-7})}{\leq} \frac{M_*\|\bx ^k-\bx ^*\|^2}{2\ell_* \|\bx^k-\bx^*\|}\rightarrow 0. \end{eqnarray}
Now we have three facts: (\ref{first-fact}), $\bx^k\rightarrow\bx^*$ from 1), and
$ \langle \nabla f(\bx^{k}), \bd^{k}\rangle  \leq -  \rho\| \bd^{k}\|^2 $ from \Cref{descent-direction}, which  together with \cite[Theorem 3.3]{facchinei1995minimization}
 allow us to claim that  eventually the step size $\alpha_k$ determined by the Armijo rule is 1,
 namely $\alpha_k=1$. Then it follows from (\ref{facts-0-5-1}) that
  \begin{eqnarray}
\label{facts-8}\|\bx^{k+1}-\bx^*\|^2&\overset{( \ref{facts-0-5-1})}{\leq}&(1-\alpha_k)\| \bx^{k}_{ T_{k} }-\bx^*_{ T_{k} }\|^2+\alpha_k \|\bx^{k}_{ T_{k} }-\bx^*_{ T_{k} }+\bd^{k}_{ T_{k} }\|^2\nonumber\\
&=&  \|\bx^{k}_{ T_{k} }-\bx^*_{ T_{k} }+\bd^{k}_{ T_{k} }\|^2
\overset{( \ref{facts-7})}{\leq} (0.5M_*/{\ell_* })^2\|\bx ^k-\bx ^*\|^4.
\end{eqnarray}
Namely, the sequence converges quadratically, 
which completes the whole proof.\qed
\end{proof}
\section{Numerical Experiments}\label{Sec: numerical exp}
\noindent In this part, we will conduct extensive numerical experiments of our algorithm \nzno\
by using MATLAB (R2019a) on a laptop of 32GB memory and Inter(R) Core(TM) i9-9880H 2.3Ghz CPU for solving the CS problems and the sparse linear complementarity problems.
\subsection{Implementation of \nzno}\label{imp}

\noindent We  initialize  \nzno\ with $\bx^0=0$ so that $\widetilde T_0$ in \eqref{Tk} is non-empty if  $\lambda\in(0,\underline\lambda)$.  first need to set up The halting conditions is set up as follows. 

{\bf (a) Halting conditions.} If a point $\bx^k$ satisfies $\supp(\bx^k)\subseteq T_k=T_{k-1}$, $\nabla_{ T_k}  f (\bx^k)=0$ and $\bx^k_{\overline{ T }_k}=0$, then similar reasoning to prove \cref{whole-sequence-converge} 2) allows us to show it is necessary a $\tau$-stationary point of \eqref{L0O} with  $0<\tau<\min_{i}\{|x_i^k|/(2\lambda), i\in\supp(\bx^k)\}.$  Therefore, it makes sense to terminate \nzno\ at $k$th step if it meets one of following conditions: I) $k$ reaches the maximum number  of iterations (e.g., 2000) or  II) $\supp(\bx^k)\subseteq T_k=T_{k-1}$ and $\| F_{\tau_k} (\bx^k; T_k) \|\leq 10^{-6}$.

{\bf (b) Selection of parameters.}   We fix $\sigma=5\times10^{-5}$ and $ \beta=0.5$. While for $\lambda$, $\delta$ and $\tau$,  the empirical numerical experience have indicated a better strategy is to update them adaptively. Note that conditions in   \Cref{converge-rate} are sufficient but not necessary. Therefore, there is no need to set parameters strictly meeting them in practice. 

More precisely, \Cref{converge-rate} states any positive $\delta \in(0,\min \{1,\ell \})$  is acceptable, but in practice to guarantee more steps with Newton directions, it is suggested to be relatively small \cite{de1996semismooth,facchinei1997nonsmooth}.
On the other side,  the condition  $0<\tau<\overline{\tau}\leq2\overline{\alpha} \delta  \beta/(nL^2)$ from (\ref{alpha-r}) 
suggests $\tau$ should  be small enough if $\delta$ is chosen to be  small. However,  $\widetilde T_{k}$ would not vary too much in \eqref{Tk} if a sufficiently small $\tau$ is selected at the beginning. This might push \nzno\ to fall in a local area rapidly,
which clearly degrades the performance of the algorithm. So, we set
 $$\delta:=\delta_k=\left\{\begin{array}{ccc}
                      10^{-10}, & {\rm if}&  S_{k} =\emptyset,  \\
                       10^{-4}, &   {\rm if}& S_{k}\neq \emptyset.
                     \end{array}
 \right.$$
In spite of that \Cref{converge-rate} has given us a clue to choose $0<\tau<\overline{\tau}$, it is still difficult to fix a proper one since ${L}$ is not easy to compute in general.  An alternative is to update $\tau$ adaptively. Typically, we use the following rule: starting $\tau$ with a fixed scalar $\tau_0$ (e.g., $\tau_0=1/2$ if no extra explanations are given) and then  update it as,
\begin{eqnarray*}
\tau_{k+1}&=&\left\{
\begin{array}{ll}
\tau_k/1.25,&\text{if}~k/10=\lceil k/10\rceil~  \text{and}\ \|F_{\tau_k}(\bx^{k};T_{k})\|>  k^{-2}, \\
\tau_k 1.25,&\text{if}~k/10=\lceil k/10\rceil~  \text{and}\ \|F_{\tau_k}(\bx^{k};T_{k})\|\leq  k^{-2}, \\
\tau_k,&\text{otherwise}.
\end{array}
\right.
\end{eqnarray*}

{\bf (c) Tuning $\lambda$.} It is suggested to set $\lambda\in(0,\underline \lambda)$ in  \Cref{framework-SNSCO} to avoid a trivial solution $0$, where $\underline \lambda$ is given by \eqref{lamda_upper}. However, $\underline \lambda$ might incur a very small $\lambda$ and thus a big size $|\widetilde T_{k}|$ by \eqref{Tk}. Note that the complexity of  deriving the Newton direction by \eqref{d-k-nonsingular} is at least about $O(|\widetilde T_{k}|^3)$. Therefore,   a small $\lambda$ not only increases the computational complexity but also results in a solution that is not sparse enough. On there other hand, as mentioned in \cref{lambda-upper-lower}, a too big value of $\lambda$ (e.g. $\lambda>\overline{\lambda}$ defined in \eqref{lamda_upper}) would result in a trivial solution $0$. To balance these two aspects, we start with a slightly bigger $\lambda_0:= \max\{\underline \lambda, c \overline \lambda\}$ and gradually reduce it by $\lambda_k=r\lambda_{k-1}$, where $r,c\in(0,1]$. We pick $r=0.75$ and $c=0.5$ in our numerical experiments if no extra explanations are provided. 
\begin{figure}[!t]
\centering
  \includegraphics[width=1\linewidth]{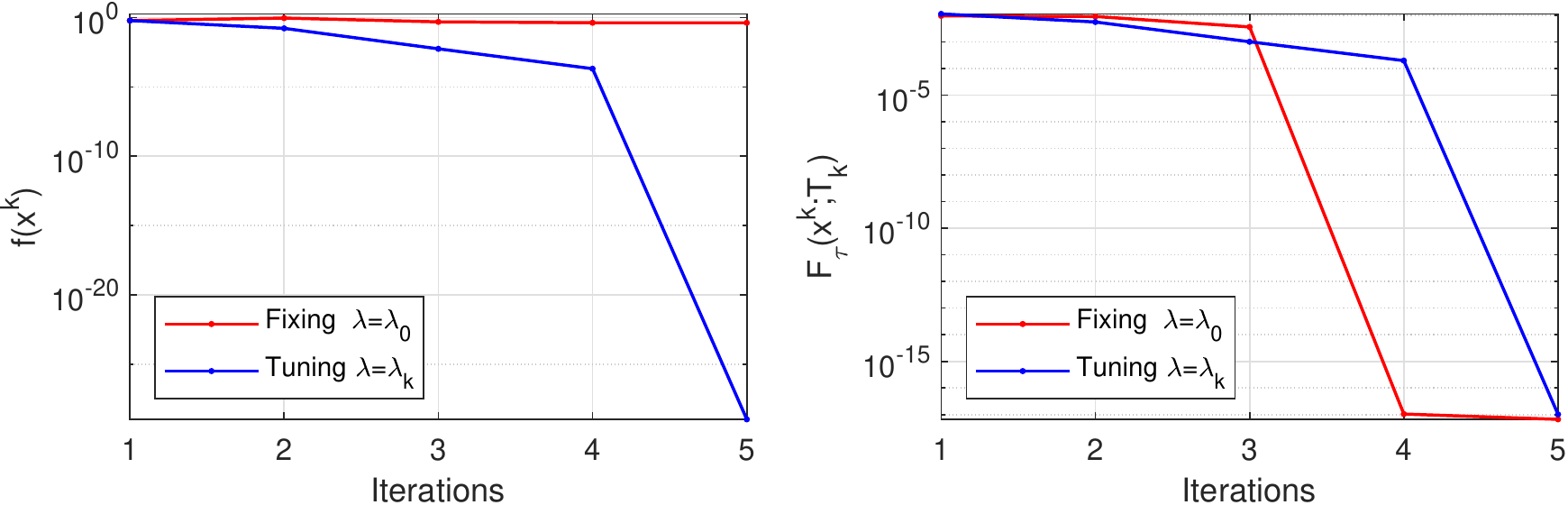}\\\vspace{3mm}
  \includegraphics[width=1\linewidth]{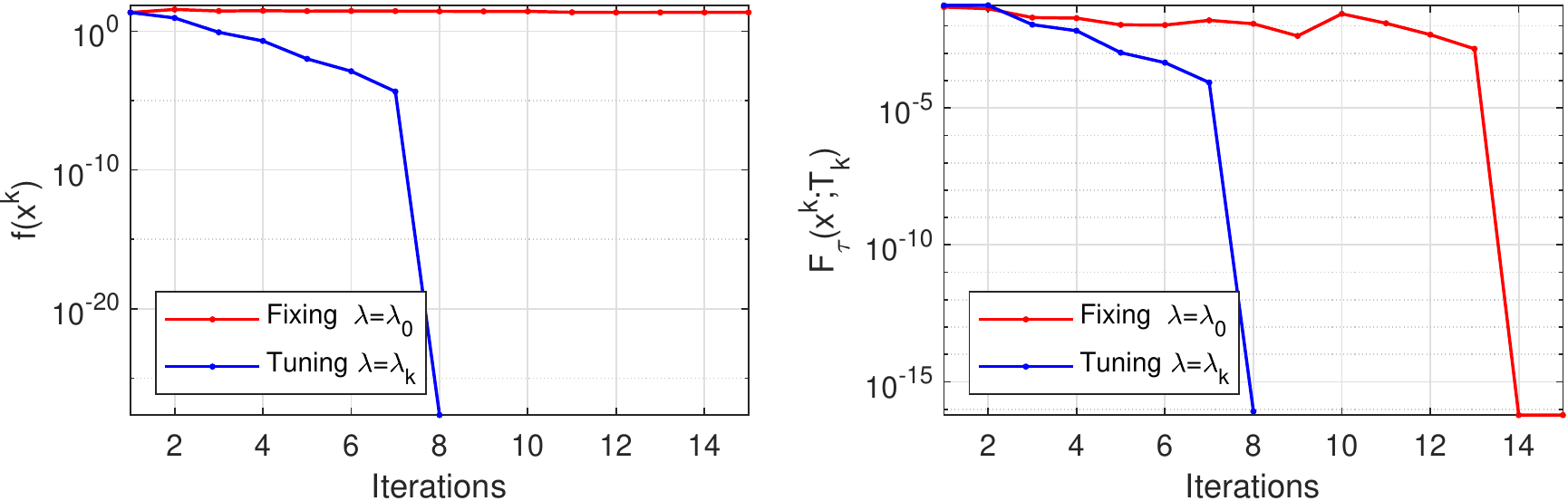} 
\caption{Two strategies for setting $\lambda$ in NL0R for solving \Cref{cs-ex}. The sub-figures in the top (bottom) row are produced by NL0R under $s_*=100$ ($s_*=500$). }  
\label{fig:tau-error}
\end{figure}

To see the performance of NL0R under fixing $\lambda=\lambda_0$ or updating $\lambda=\lambda_k$,  two instances of \Cref{cs-ex}  are tested and according results are shown in \cref{fig:tau-error}. It can be clearly seen that $\| F_{\tau_k} (\bx^k; T_k) \|$ declines dramatically for both fixing $\lambda=\lambda_0$ and updating $\lambda=\lambda_k$, indicating NL0R enjoys a quadratic convergence property. While the objective $f(\bx^k)$ produced by NL0R under fixing $\lambda=\lambda_0$ stabilizes at a level, which means it achieves a local minima. By contrast, NL0R under updating $\lambda=\lambda_k$ delivers the objective $f(\bx^k)$ that drops down sharply and approaches to a globally optimal value. Therefore, the updating rule makes NL0R perform better and thus is adopted to proceed with our numerical comparisons in the sequel.

\subsection{Compressed  {sensing}}
\noindent CS  has seen revolutionary advances both in theory and
 {algorithm} over the past decade. Ground-breaking papers that pioneered the advances are \cite{donoho2006compressed, candes2006robust,candes2005decoding}.
We will focus on two types of data: the randomly generated data and the 2-dimensional image data. For the first data, we consider the exact recovery $\by=A \bx$, where the sensing matrix $A$  chosen as in \cite{yin2015minimization,zhou2016null}. While for the image data, we consider the inexact recovery $\by=A \bx +\xi$, where $\xi$ is the noise and $A$ will be described in \Cref{cs-imag}.
\begin{example}[Random data]\label{cs-ex} Let $A\in\mathbb{R}^{m\times n}$  be a random Gaussian matrix  with each column   being identically and independently distributed  (iid) samples of the standard normal distribution. We then normalize each column to be a unit length. Next, the $s_*$ non-zero components of the `ground  truth' signal $\bx^*$ are also iid samples of the standard normal distribution, and their locations are picked randomly. Finally, the measurement  is given by $\by=A \bx^*$. 
\end{example}

\begin{example}[2-D image data]\label{cs-imag} Some images are naturally not sparse themselves but could be sparse under some wavelet transforms. Here, we take advantage of the  Daubechies wavelet 1, denoted as $W(\cdot)$.  Then images under this transform (i.e., $x^*:=W(\omega)$) is sparse,  $\omega$ be the vectorized intensity of an input image. Because of this, the explicit
form of  the sampling matrix may not be available. We consider a sampling matrix taking the form $A=FW^{-1}$, where $F$ is the partial fast Fourier transform, and $W^{-1}$ is the inverse of $W$. Finally, the added noise $\xi$  has each element $\xi_i\sim {\tt nf }\cdot\mathcal{N}$ with $\mathcal{N}$ being the standard normal distribution and ${\tt nf }$ being the noise factor. Three typical choices of ${\tt nf }$ are taken into account, namely ${\tt nf }\in\{0.01,0.05,0.1\}$. For this experiment, we compute  a gray image (see the original image in  $\Cref{fig:image-1}$) with size $512\times512$ (i.e. $n=512^2=262144$) and the sampling size $m=20033$ and $29729$ respectively.

\end{example}
\subsubsection{Comparisons for random data}
\noindent Since a large number of state-of-the-art methods have been proposed to solve the CS problems, it is far beyond our scope to compare all of them. To make comparisons fair, we  only focus on those algorithms (often referred as regularized methods) which aim at solving (\ref{L0O}) or its relaxations, where $\ell_0$ norm is replaced by some approximations such as  $\ell_q (0<q\leq1)$  \cite{lai2013improved} or $\ell_1-\ell_2$ \cite{lou2018fast}.  Note that greedy methods mentioned in Subsection \ref{subsec:lit-rev}, for the model \eqref{SCO} with $s$ being given, have been famous for the super-fast computational speed  and the high order of accuracy when $s$ is relatively small to $n$.  However, we will not compare them with \nzno\ since we would like to consider the scenario when $s$ is unknown. We select MIRL1 \cite{zhou2016null}, AWL1 \cite[ADMM for weighted $\ell_{1-2}$]{lou2018fast}  which is a faster approximation of the method proposed in \cite{yin2015minimization}, IRSLQ  \cite{lai2013improved} (we choose $q=1/2$) and PDASC \cite{jiao2015primal}. All parameters are set as default except for setting the maximum iteration number as 100 and removing the final refinement step for MRIL1 and \texttt{del=1e-8} for PDASC. Note that PDASC and \nzno\ are the second-order methods and the other  three belong to the category of the first-order methods.



To see the accuracy of the solutions and the speed of these five methods, we run 20 trials with medium dimensions $n$ increasing from 10000 to 30000 and keeping $m=\lceil 0.25n\rceil, s_*=\lceil 0.01n\rceil$ or $s_*=\lceil 0.05n\rceil$. Average results  are reported in \Cref{fig:aver-error-time}, where $s_*=\lceil 0.01n\rceil$, and \cref{tab:cs-1}, where $s_*=\lceil 0.05n\rceil$. As shown in \Cref{fig:aver-error-time},
\nzno\ always  generates  the smallest $\|\bx-\bx^*\|$, the most accurate recovery, with accuracy order at least $10^{-14}$, followed by PDSAC.  By contrast, the other three methods  {get} accuracy with the order being above $10^{-5}$. This phenomenon well testifies that the second-order methods have their advantages in producing a higher order of accuracy. When it comes to the computational speed, it can be clearly seen   that \nzno\ always runs the fastest, with only consuming about 2 seconds when $n=30000$.
PDSAC  is  the runner up. This shows that, for problems in higher dimensions,   \nzno\ and PDSAC are able to run  faster than the first-order methods. Similar observations can be seen in \cref{tab:cs-1}. In a nutshell, \nzno\ delivers  the most accurate recovery within the shortest computational time.

\begin{figure}[!b]
\centering
\begin{subfigure}{0.495\textwidth}
  \includegraphics[width=.9\linewidth]{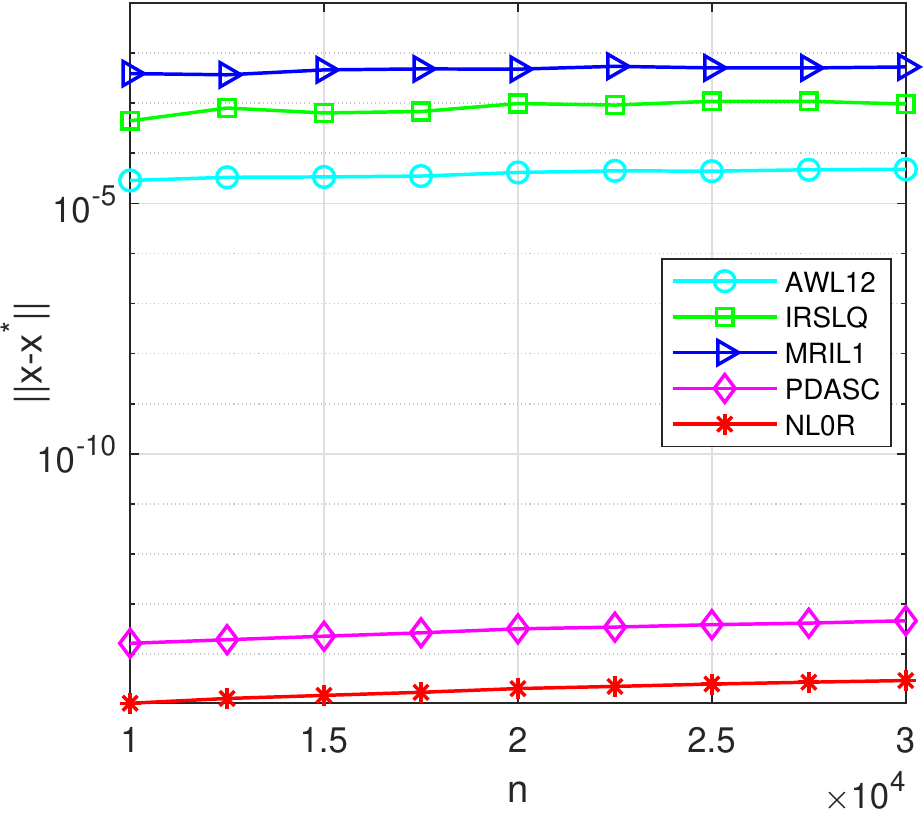}
\end{subfigure}
\begin{subfigure}{0.49\textwidth}
  \includegraphics[width=.9\linewidth]{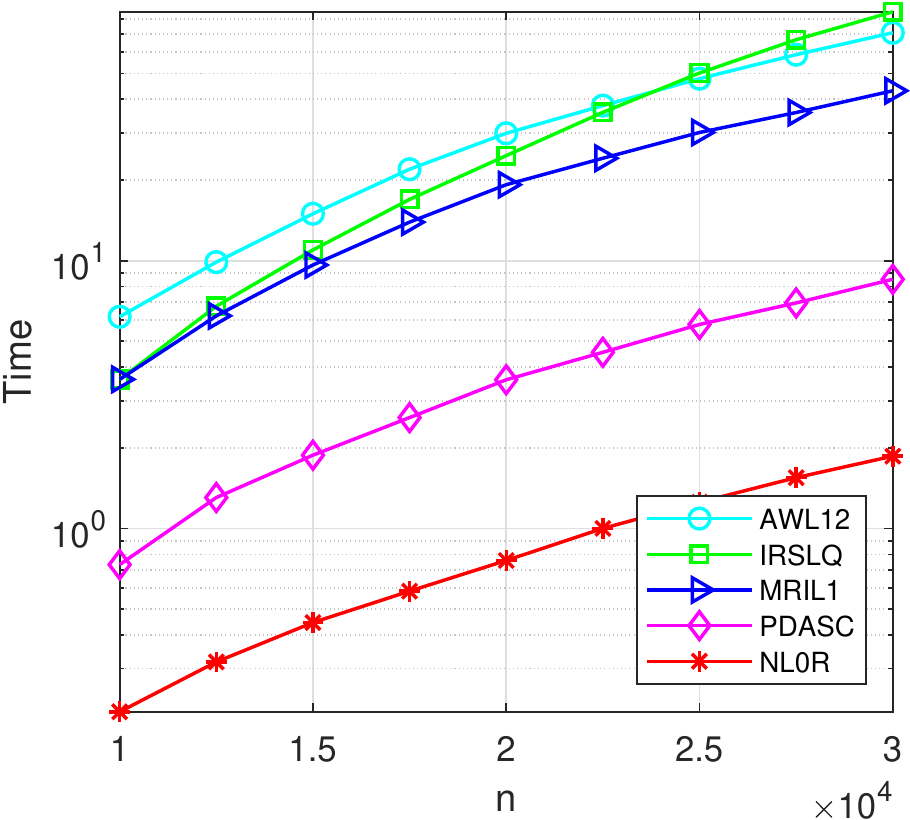}
\end{subfigure} 
\caption{Average recovery error and time of five methods for solving \Cref{cs-ex}.} 
\label{fig:aver-error-time}
\end{figure}

\begin{table}[!th]
 \caption{Performance of five methods for \Cref{cs-ex}. \label{tab:cs-1}}\vspace{-3mm}
{\renewcommand{\arraystretch}{1}\addtolength{\tabcolsep}{-2.9pt}
{\centering
\begin{tabular}{ lccccccccccc} \\ \hline
& \multicolumn{5}{c}{$\|\bx-\bx^*\|$}&& \multicolumn{5}{c}{Time (in seconds)}\\ \cline{2-6}\cline{8-12}
n&10000&15000&20000&25000&30000 &&10000&15000&20000&25000&30000 \\\hline 
AWL12	&	8.39e-05	&	1.10e-04	&	1.21e-04	&	1.32e-04	&	1.43e-04	&	&	17.71 	&	42.70 	&	85.46 	&	133.3 	&	195.4 	\\
RSLQ 	&	3.79e-04	&	4.32e-04	&	4.05e-04	&	3.58e-04	&	6.25e-04	&	&	7.653 	&	23.84 	&	56.38 	&	113.1 	&	189.3 	\\
MRIL1	&	1.57e-02	&	1.96e-02	&	2.48e-02	&	2.63e-02	&	2.54e-02	&	&	4.595 	&	12.00 	&	23.21 	&	36.93 	&	52.36 	\\
PDASC	&	5.36e-14	&	7.81e-14	&	1.07e-13	&	1.33e-13	&	1.59e-13	&	&	0.972 	&	2.290 	&	4.680 	&	7.514 	&	11.12 	\\
SNL0	&	1.16e-14	&	6.58e-15	&	2.37e-14	&	2.96e-14	&	3.55e-14	&	&	0.602 	&	1.363 	&	2.549 	&	4.175 	&	6.303 	\\
\hline	
 \end{tabular}\par} }
\end{table}

\subsubsection{Comparisons for 2-D image data}
\noindent	In \Cref{cs-imag}, data size $n$ is relatively large,
which possibly makes most regularized methods  suffer extremely slow computation. Hence, we  select three greedy methods CSMP (denoted for CoSaMP) \cite{needell2009cosamp}, HTP \cite{foucart2011hard} and  AIHT \cite{blumensath2012accelerated} as well as PDSCA. As suggested in package PDSCA, we set another rule to stop each method  if at $k$th iteration  it satisfies $\|A\bx^k-y\|\leq \|A\bx^*-y\|$ to speed up the termination. Moreover, to make comparisons fair, we fist run PDSCA, which is capable of delivering a solution with a good sparsity level $s$.  Then  we set this sparsity level  $s$ for CSMP, HTP and  AIHT since they need such prior information. Let $\bx$ be a solution produced by a method. Apart from reporting the sparsity level $\|\bx\|_0$ and the CPU time of a method, we also compute the  peak signal to noise ratio (PSNR) defined by
$${\rm PSNR}:=10\log_{10}(n\|\bx-\bx^*\|^{-2})$$
to measure the performance of the method. Note that the  larger  PSNR is,  the much closer $\bx$  approaches to the true image $\bx^*$, namely the better performance of a method yields. Results for  \Cref{cs-imag} are presented in
\Cref{fig:image-1} and \Cref{tab:cs-imag-1} , where
SPDSA offers the biggest PSNR when {\tt nf}$=0.01$, whilst \nzno\ produces the biggest  ones  when {\tt nf}$=0.05$ and {\tt nf}$=0.1$, which means our method is more robust to the noise. In addition, \nzno\  runs the fastest and renders the sparsest representations for most cases. 

\begin{figure}[!th]
\centering 
  \includegraphics[width=.95\linewidth]{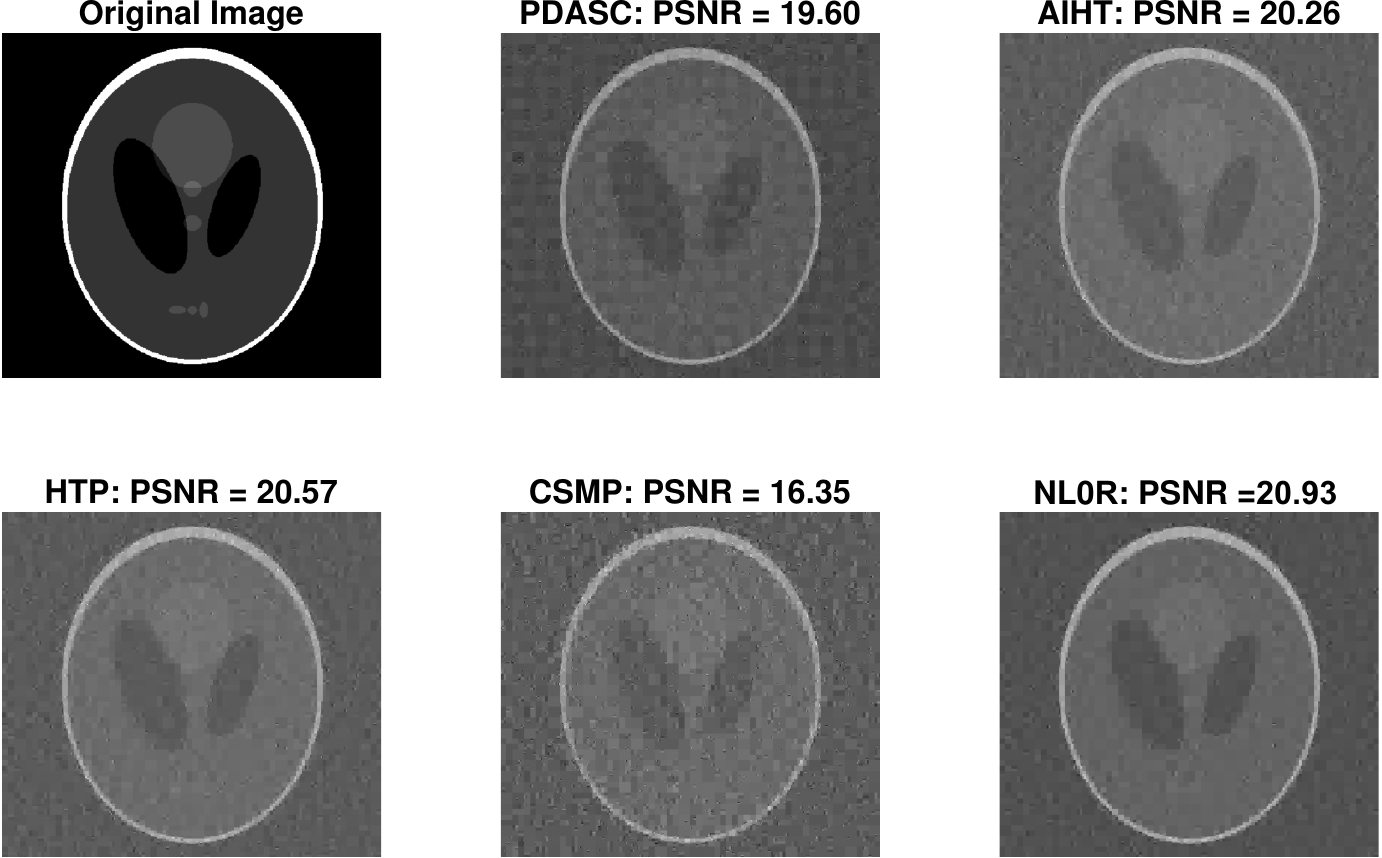}
\caption{Recovery results for  \Cref{cs-imag} with $m=20033$ and {\tt nf}$=0.1$.} 
\label{fig:image-1}
\end{figure}

\begin{table}[!th]
 \caption{Performance of five methods for \Cref{cs-imag}. \label{tab:cs-imag-1}}\vspace{-3 mm}
{\renewcommand{\arraystretch}{1}\addtolength{\tabcolsep}{-1.5pt}
{\centering\begin{tabular}{ lllcc ccccccccc}\\ \hline
&&\multicolumn{3}{c}{ {\tt nf}$=0.01$}    &&\multicolumn{3}{c}{{\tt nf}$=0.05$} &&\multicolumn{3}{c}{{\tt nf}$=0.1$}\\\cline{3-5}\cline{7-9}\cline{11-13}
&&PSNR&Time&$\|\bx\|_0$&&PSNR&Time&$\|\bx\|_0$&&PSNR&Time&$\|\bx\|_0$\\\hline 
&	SPDSA	&	21.62 	&	15.53 	&	9716	&&	20.11 	&	8.45 	&	5982	&&	19.60 	&	5.72 	&	2969	\\
$m=20033$
&	AIHT	&	19.81 	&	148.5 	&	9716	&&	20.15 	&	2.23 	&	5982	&&	20.26 	&	19.3 	&	2969	\\
$n=262144$
&	HTP	&	19.66 	&	19.15 	&	9716	&&	20.27 	&	3.40 	&	5982	&&	20.57 	&	3.41 	&	2969	\\
&	SCMP	&	12.49 	&	51.54 	&	9716	&&	18.44 	&	63.1 	&	5982	&&	16.35 	&	14.8 	&	2969	\\
&	NL0R	&	23.21 	&	7.130 	&	9690	&&	21.91 	&	4.43 	&	4173	&&	20.93	&	3.07 	&	2803\\\hline	
&SPDSA	&	35.37 	&	11.54 	&	9902	&&	25.07 	&	6.58 	&	5002	&&	22.61 	&	5.44 	&	3513	\\
$m=29729$
&	AIHT	&	32.21 	&	71.42 	&	9902	&&	24.78 	&	9.52 	&	5002	&&	23.07 	&	9.16 	&	3513	\\
$n=262144$&	HTP	&	34.89 	&	14.38 	&	9902	&&	25.14 	&	4.57 	&	5002	&&	23.19 	&	2.02 	&	3513	\\
&	SCMP	&	21.48 	&	39.79 	&	9902	&&	23.00 	&	9.94 	&	5002	&&	20.73 	&	2.26 	&	3513	\\
&	NL0R	&	33.59 	&	6.761 	&	8787	&&	25.31 	&	3.99 	&	 3885	&&	23.23 	&	2.58 	&	2641	\\\hline
 \end{tabular}\par} }
\end{table}

\subsection{Sparse linear complementarity problem}
\noindent Sparse linear complementarity problems have been applied into dealing with real-world applications such as bimatrix games and portfolio selection problems \cite{CPS92,XHZ08, SZX14}. The problem aims at finding a sparse vector $\bx\in\R^n$ from $\Omega:=\{\bx\in\R^n:~\bx\geq0,\  M\bx+q \geq0, \ \langle \bx, M\bx+q\rangle=0 \},$
where $M\in\R^{n\times n}$ and $q\in\R^{n}$. A point $\bx\in \Omega$ is equivalent to
\begin{eqnarray}\label{sco-obj}
f(x)&:=&\sum_{i=1}^n \phi(x_i,M_i\bx+q_i)=0,
\end{eqnarray}
where $\phi$ is the so-called NCP function, which is defined by $\phi(a,b)=0$ if and only if $a\geq0,b\geq0,ab=0$. We take advantage of an NCP function $\phi(a,b)=a_+^2b_+^2+(-a)_+^2+(-b)_+^2$, where $a_+:=\max\{a,0\}$, and a testing example from \cite{zhou2020newton}.
  \begin{example}\label{sdp-matrix}   Let $M=ZZ^\top$ with $Z\in \R^{n\times m}$ and $m\leq n$ (e.g. $m=n/2$). Elements of $Z$ are  iid samples from the standard normal distribution. Each column is then normalized to have a unit length. The `ground truth' sparse solution $\bx^*$ with a sparsity level $s_*$ is produced the same as in \cref{cs-ex} and $q$ is obtained by
$q_i= -(M \bx^*)_i$ if  $x_i^*>0$ and $q_i=|(M \bx^*)_i|$ otherwise.
 \end{example}

Since there are very few methods that have been proposed to process the sparse LCP, we only select two solvers: the half-thresholding projection (HTP) method \cite{shang2015extragradient} and LEMKA's method (LEMKE\footnote{\url{http://ftp.cs.wisc.edu/math-prog/matlab/lemke.m}}). We alter the sample size $n$ but fix $m=n/2, s_*=0.01n$ and $s_*=0.05n$. Average results over 20 trials are reported in \cref{fig:s-1} where $s_*=0.01n$ and \cref{tab:lcp} where $s_*=0.05n$. Comparing with HTP, LEMKE and NL0R produce much more accurate solutions since their obtained objective function values $f(\bx)$ and the recovered accuracy $\|\bx-\bx^*\|$ almost tend  to zero. When it comes to the computational speed, the picture is significantly different. As shown in \cref{fig:s-1}, NL0R runs super-fast, followed by LEMKE, and HTP comes the last. Similar observations can be seen in \cref{tab:lcp}, where for the case of $n=20000$, NL0R only consumes about 8.826 seconds  while  LEMKE takes 531.1 seconds and HTP needs 207.9 seconds. Therefore, NL0R evidently outperforms the others in the high dimensional settings.

 \begin{figure}[H]
\centering
\begin{subfigure}{0.325\textwidth}
  \includegraphics[width=1\linewidth]{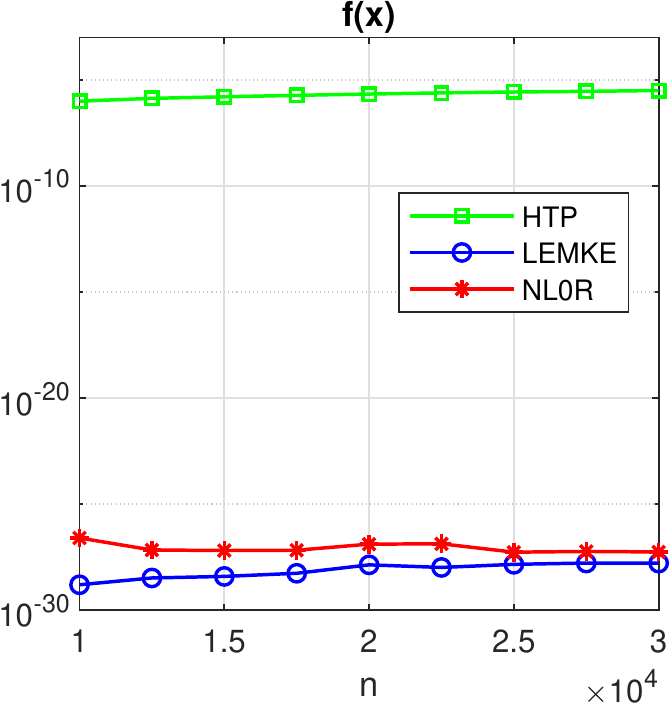}
\end{subfigure}
\begin{subfigure}{0.325\textwidth}
  \includegraphics[width=1\linewidth]{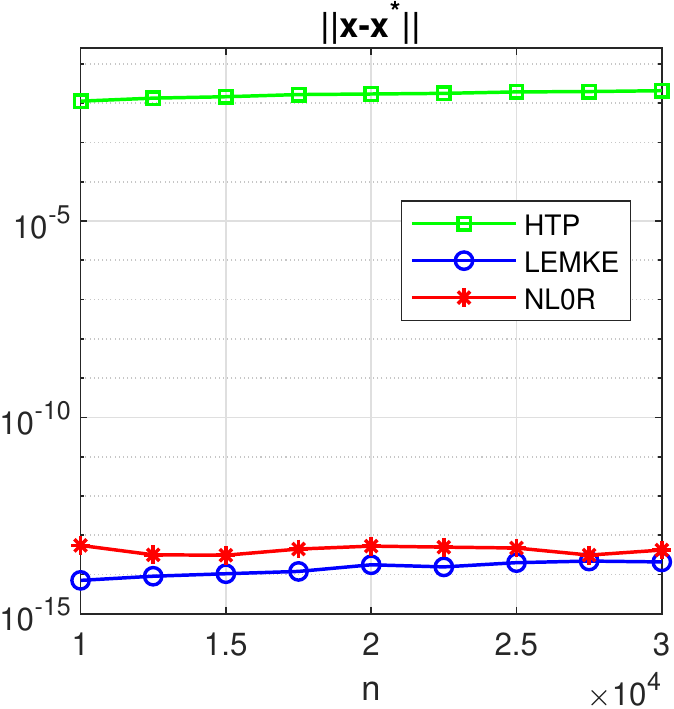}
\end{subfigure} 
\begin{subfigure}{0.325\textwidth}
  \includegraphics[width=.95\linewidth]{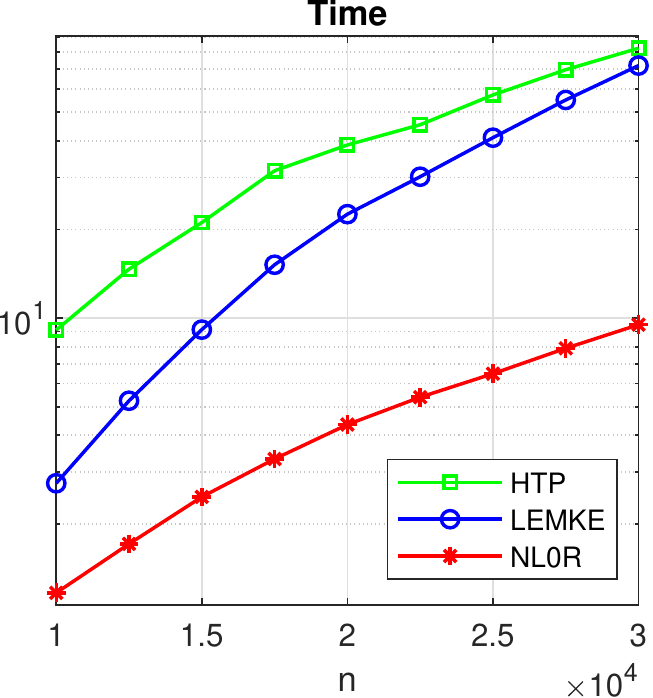}
\end{subfigure} 
\caption{Performance of \nzno\ effected by $\lambda_0$ for solving \Cref{sdp-matrix}.} 
\label{fig:s-1}
\end{figure}


\begin{table}[!th]
 \caption{Performance of three methods for \Cref{sdp-matrix}. \label{tab:lcp}}\vspace{-3mm}
{\renewcommand{\arraystretch}{1}\addtolength{\tabcolsep}{-2.9pt}
{\centering
\begin{tabular}{ lccccccccccc} \\ \hline
& \multicolumn{3}{c}{$f(\bx)$}&&\multicolumn{3}{c}{$\|\bx-\bx^*\|$}&& \multicolumn{3}{c}{Time (in seconds)}\\ \cline{2-4}\cline{6-8}\cline{10-12}
$n$	&	HTP	&	LEMKE	&	NL0R	&	&	HTP	&	LEMKE	&	NL0R	&	&	HTP	&	LEMKE	&	NL0R	\\\hline
5000	&	2.52e-06	&	1.15e-28	&	1.94e-27	&	&	5.55e-02	&	2.05e-14	&	6.46e-14	&	&	11.83 	&	7.911 	&	0.581 	\\
7500	&	4.20e-06	&	3.22e-28	&	3.84e-28	&	&	7.04e-02	&	3.15e-14	&	4.01e-14	&	&	27.69 	&	27.14 	&	1.240 	\\
10000	&	5.38e-06	&	7.21e-28	&	3.33e-27	&	&	8.36e-02	&	4.58e-14	&	9.88e-14	&	&	50.71 	&	64.64 	&	2.216 	\\
12500	&	6.76e-06	&	9.06e-28	&	4.00e-28	&	&	8.87e-02	&	5.10e-14	&	3.18e-14	&	&	79.96 	&	127.7 	&	3.434 	\\
15000	&	7.99e-06	&	1.18e-27	&	8.83e-28	&	&	9.86e-02	&	6.80e-14	&	6.07e-14	&	&	114.7 	&	221.1 	&	4.994 	\\
17500	&	9.30e-06	&	2.19e-27	&	8.37e-28	&	&	1.08e-01	&	7.69e-14	&	4.23e-14	&	&	158.4 	&	354.0 	&	6.862 	\\
20000	&	1.12e-05	&	3.22e-27	&	9.71e-27	&	&	1.18e-01	&	1.10e-13	&	2.72e-13	&	&	207.9 	&	531.1 	&	8.826 	\\

\hline	
 \end{tabular}\par} }
\end{table}

\section{Conclusion}
\noindent A vast body of work has developed numerical methods that only make use of the first-order information of the involved functions. Because of this, they are able to run fast but suffer from slow convergence. When Newton steps are integrated into some of these methods, then much more rapid convergence can be achieved. To the best of our knowledge, the current theoretic guarantees include two groups: either the (sub)sequence converges to a stationary point of $\ell_0$-regularized optimization or the distance between each iterate and any given sparse reference point is bounded by an error bound in the sense of probability.  However, those do not thoroughly reveal the reasons why those methods with Newton steps perform exceptionally well. In this paper,  we designed a Newton-type method for the $\ell_0$-regularized optimization and proved that the generated sequence converges to a stationary point globally and quadratically. This well explains such a method is expected to enjoy an appealing performance from the theoretical perspective, which was testified by the numerical experiments where it is capable of rendering a relatively high order of accuracy with fast computational speed.

%

\section*{Acknowledgements}
\noindent This work was funded by the the National Science Foundation of China (11971052, 11801325, 11771255) and Young Innovation Teams of Shandong Province (2019KJI013).

\bibliographystyle{abbrv}
\bibliography{references}
\end{document}